\documentclass[reqno]{amsart}
\usepackage{hyperref}
\usepackage{mathrsfs}

\begin{document}
\title[Reflecting Ornstein-Uhlenbeck stochastic process]
{Infinite dimensional reflecting Ornstein-Uhlenbeck stochastic process}

\author[K. Akhlil]
{Khalid Akhlil}

\address{}

%\author[M. K\"unze] {Markus K\"unze}

%\address{}

%\dedicatory{Dedicated to...}

\date{\today}
\thanks{The author is supported by ``Deutscher Akademischer Austausch Dienst``(German Academic Exchange Service).}
\subjclass[2000]{subjclass}

\begin{abstract}
In this article we introduce the Gaussian Sobolev space
$W^{1,2}(\mathscr O,\gamma)$, where $\mathscr O$ is an arbitrary
open set of a separable Banach space $E$ endowed with a
nondegenerate centered Gaussian measure $\gamma$. Moreover, we
investigate the semimartingale structure of the infinite dimensional
reflecting Ornstein-Uhlenbeck process  for open sets of the form
$\mathscr O=\{x\in E\, :\,  G(x)<0\}$, where $ G$ is some Borel
function on $E$.
\end{abstract}

\maketitle
\setlength{\textheight}{19.5 cm} \setlength{\textwidth}{12.5 cm}
\newtheorem{theorem}{Theorem}[section]
\newtheorem{lemma}[theorem]{Lemma}
\newtheorem{proposition}[theorem]{Proposition}
\newtheorem{corollary}[theorem]{Corollary}
\newtheorem{Hypo}[theorem]{Assumption}
\theoremstyle{definition}
\newtheorem{definition}[theorem]{Definition}
\newtheorem{example}[theorem]{Example}
\theoremstyle{remark}
\newtheorem{remark}[theorem]{Remark}
\numberwithin{equation}{section} \setcounter{page}{1}

% MATH -------------------------------------------------------------------
\newcommand{\mo}{\mathscr O}
\newcommand{\obar}{\overline{\mo}}
\newcommand{\Dho}{D_H^{\mo}}
\newcommand{\me}{\mathscr E}
\newcommand{\EO}{\mathscr E_{\mo}}
\newcommand{\Cap}{\mathfrak{Cap}_{\obar}}
\newcommand{\Ecap}{\mathfrak{Cap}}
%\newcommand{\cap}{\mathsf{Cap}_{\obar}}
%%%%%%%%%%%%%%%%%%%%%%%%%%%%%%%%%%%%%%%%%%%%55

\section{Introduction}
Let $E$ be a separable Banach space endowed with a nondegenerate
centered Gaussian measure $\gamma$ and $H(\gamma)$ be its relevant
Cameron-Martin space, which is known to be continuously and densely
embedded in $E$. In a remarkable paper \cite{CL}, Sobolev spaces of
real valued functions defined on open sets was introduced for open
sets $\mo$ of the form $\mo=\{x\in E\,|\,G(x)<0\}$ for suitable
$G:E\rightarrow\mathbb R$. More precisely, the Sobolev spaces
$W^{1,p}(\mo,\gamma)$ are defined as the closure, in the sobolev
norm, of the operator $\Dho:\mathrm{Lip}(\mo) \longrightarrow
L^p(\mo,\gamma;H)$ defined by
 \[
 \Dho\varphi:=D_H\breve{\varphi}_{|\mo}
 \]
 where $D_H$ is the derivative in the direction of $H$ and $\breve\varphi$ is any extension of $\varphi$ to an
 element of $\mathrm{Lip}(E)$ ( $\mathrm{Lip}(\mo)$ (resp. $\mathrm{Lip}(E)$) is the space of Lipschitz continuous
 functions on $\mo$ (resp. $E$)).

 After defining Sobolev spaces $W^{1,p}(\mo,\gamma)$, the authors in \cite{CL} defined
 the trace operator $\mathrm{Tr}$ of functions in $W^{1,p}(\mo,\gamma)$ at $\partial\mo$ and proved the following
 integration by parts formula, under some ''natural`` assumptions on $G$
  \begin{equation}
 \int_{\mo}D_k^{\mo}\varphi d\gamma=\int_{\mo}\hat{v}_k\varphi d\gamma+
 \int_{\partial\mo}\frac{D_H^k G}{|D_H G|_H}\mathrm{Tr}\varphi d\rho
 \end{equation}
for every $\varphi\in W^{1,p}(\mo,\gamma)\, (p>1)$, where $\{v_k\,|\,k\in\mathbb N\}$ is an orthonormal basis of
$H(\gamma)$ and $\hat v_k$ is the element generated by $v_k$ (see subsection 2.2).

Now let $\mo$ be an arbitrary open set of $E$. In particular,
$\obar$, with the topology induced by the one of the separable
Banach space $E$, is a Luzin topological space and functions in
$L^p(\mo,\mathscr B(\mo),\gamma_{|\mo})$ have to be seen as
functions in $L^p(\obar,\mathscr B(\obar), m)$ where $m$ is defined,
for $A\in\mathscr B(\obar)$, by $m(A)=\gamma(A\cap\mo)$. In a paper
in preparation \cite{KS}, Sobolev spaces $W^{1,2}(\mo,\gamma)$ are
defined with the same procedure but for arbitrary open sets $\mo$ of
$E$ by using another method to prove the closability of $\Dho$ (see
Lemma \ref{lem:closable}). Moreover, a relative Gaussian capacity of
sets in $\obar$ is introduced. It is the capacity associated with
the Dirichlet form $(\EO, D(\EO))$ on $L^2(\obar,m)$ with domain
$D(\mathscr E)=W^{1,2}(\mo,\gamma)$ defined by
 \begin{equation}\label{eq:form}
\EO (\varphi,\psi) = \int_\mo [D_H^\mo \varphi, D_H^\mo \psi]_H\, d\gamma
\end{equation}

The Gaussian relative capacity is a Choquet capacity and is tight,
which means that the Dirichlet form $(\EO, D(\EO))$ is
quasi-regular. Moreover, it is local and hence its associated right
process $\mathbf M=(\Omega, \mathscr F,(X_t)_{t\geq 0},(P_z)_{z\in
E})$ is, in fact, a diffusion process.

The purpose of this paper is to prove, for open sets of the form
$\mo=\{x\in E\,|\,G(x)<0\}$ for suitable $G:E\rightarrow\mathbb R$,
that the diffusion process $(X_t)_{t\geq 0}$ associated with $(\EO,
D(\EO))$ is a semimartingale with a Skorohod type decomposition. As
in the finite dimensional framework, we will use the well-known
Fukushima decomposition, which holds in the situation of
quasi-regular Dirichlet forms by using the transfer method. For a
relatively quasi-continuous
 $\gamma-$version $\tilde{\varphi}$ of $\varphi\in W^{1,2}(\mo,\gamma)$, the additive functional
 $(\tilde{\varphi}(X_t)-\tilde{\varphi}(X_0))_{t\geq 0}$ of $\mathbf M$ can uniquely be represented as

  \[
  \tilde{\varphi}(X_t)-\tilde{\varphi}(X_0)=M_t^{[\varphi]}+N_t^{[\varphi]}, t\geq 0
  \]
 where $M^{[\varphi]}:=(M_t^{[\varphi]})_{t\geq 0}$ is a MAF of $\mathbf M$ of finite energy and
 $N^{[\varphi]}:=(N_t^{[\varphi]})_{t\geq 0}$ is a CAF of $\mathbf M$ of zero energy.

To evaluate the bracket $\langle M^{[\varphi]}\rangle$ of the
martingale additive functional $M^{[\varphi]}$ for $\varphi\in
W^{1,2}(\mo,\gamma)$, we use a standard technic as for the finite
dimensional case \cite{BHs} and used in the infinite dimensional
framework in the case $\mo=E$ for a more general $E$( see for
example \cite[Proposition 4.5]{AR}). Let $\varphi\in
W^{1,2}(\mo,\gamma)$, then one obtains
\begin{equation}
\langle M^{[\varphi]}\rangle_t=\int_0^t[\Dho\varphi(X_s),\Dho\varphi(X_s)]_Hds, t\geq 0
\end{equation}
To evaluate $N^{[\varphi]}$ we shall characterize, as in the regular
Dirichlet forms framework, the boundedness of its variation which is
an easy task by using the transfer method (see Lemma \ref{lem:bv}).

To simplify our calculus, we consider two identifications: The
standard one consisting of identifying $H(\gamma)$ with its dual
$H(\gamma)'$ and the second consisting of identifying $E'\times
H(\gamma)$ with $H(\gamma)\times H(\gamma)$, which means that one
consider the dualisation $~_{E'}\langle,\rangle_E$ to coincide with
$[,]_H$ when restricted to $E'\times H(\gamma)$. In this situation
one obtain a countable subset $K_0=\{l_k,\,k\in\mathbb N\}$ of $E'$
forming an orhonormal basis of $H(\gamma)$ and separating the points
of $E$ such that the linear span $K\subset E'$ of $K_0$ is dense in
$H(\gamma)$.

Our first result consists of componentwise semimartingal structure
of $\mathbf M$. We define the following coordinate functions: For
$l\in K$, with $|l|_H=1$, define
\[
\varphi_l(z)=~_{E'}<l,z>_E,\quad z\in E
\]
For this functions, Fukushima decomposition becomes as follow:
\begin{equation}
\varphi_l(X_t)-\varphi_l(X_0)=W^l_t+\int_0^t\hat l(X_s)ds+\int_0^t\nu^l_{ G}(X_s)dL^{\rho}_s
\end{equation}
where for all $z\in\obar\setminus S_l$ for some relative polar set
$S_l\subset\obar$ the continues martingale $(W_t^l,\mathscr F_t,
P_z)_{t\geq 0}$ is a one dimensional Brownian motion starting at
zero and $\hat l$ is the element generated by $l$. The vector
$\nu^l_G$ is defined by
\[\nu^l_G=\frac{D_H^lG}{|D_H G|_H}\]
plays the role of the outward normal vector field in the direction
of $l$ and $L^{\rho}_t$ is the positive continuous additive
functional associated with the Gaussian-Hausdorff meausre $\rho$ by
Revuz correspondence.

After surrounding some technical problems we will be able to prove our second main result. It says that there exists
always a map $W:\Omega\rightarrow C([0,\infty[,E)$ such that for r.q.e. $z\in\obar$ under $P_z$, $W=(W_t)_{t\geq 0}$
is an $(\mathscr F_t)_{t\geq 0}-$Brownian motion on $E$ starting at zero with covariance $[,]_H$ such that for r.q.e.
$z\in\obar$
 \begin{equation}\label{eq:skorohod2}
  X_t=z+W_t+\int_0^tX_s\,ds+\int_0^t\nu_G(X_s)\,dL^{\rho}_s
 \end{equation}
where $L^{\rho}_t:=(L^{\rho}_t)_{t\geq 0}$ is as before and $\nu_G$ is a unite vector defined by
\[
 \nu_G:=\frac{D_H G}{|D_H G|_H}.
\]

Such results of semimartingale structure of the reflecting
Ornstein-Uhlenbeck stochastic process were already considered but
for the space of BV functions and for a very smooth sets, namely
convex sets( see \cite{BPT}, \cite{BPT2}, \cite{RZZ}, \cite{FH} and
references therein). The paper \cite{CL} opens a new perspectives on
dealing with open sets in infinite dimensions framework, in
particular for the infinite dimensional reflecting
Ornstein-Uhlembeck stochastic process as developed in the current
paper.

 %%%%%%%%%%%%%%%%--NEW SECTION--%%%%%%%%%%%%%%%%%%%%%%%%%
\section{Preliminaries}

In this section we recall some facts about the theory of
quasi-regular Dirichlet forms and the associated right processes. It
is the adequate framework when one want to deal with Sobolev spaces
in infinite dimensions, but one cannot either use directly the
general theory of Dirichlet forms as described in \cite{F80}.
However it is possible to transfer our framework in the situation of
\cite{F80} by using a compactification method (see \cite{MR} for
more details). A second element to introduce is the theory of
Gaussian measures as summarized in \cite{Bo}.

\subsection{Quasi-regular Dirichlet forms}

Let $\mathcal H$ be a real Hilbert space with inner product $(,)_H$ and norm $\|.\|_H$. Let $\mathcal D$ be a
linear subspace of $\mathcal H$ and $\mathscr E:\mathcal D\times\mathcal D\rightarrow\mathbb R$ a bilinear map.
Assume that $(\mathscr E,\mathcal D)$ is positive definite (i.e. $\mathscr E(u):=\mathscr E(u,u)\geq 0$ for all
$u\in\mathcal D$). Then $(\mathscr E,\mathcal D)$ is said to satisfy \textit{the weak sector condition} if, there
exists a constant $K>0$, called continuity constant, such that \[|\mathscr E_1(u,v)|\leq K\mathscr E_1(u,u)^{1/2}
\mathscr E_1(v,v)^{1/2}\] for all $u,v\in\mathcal D$. A pair $(\mathscr E, D(\mathscr E))$ is called a
\textit{coercive closed form} on $\mathcal H$ if $D(\mathscr E)$ is a dense linear subspace of $\mathcal H$ and
the bilinear map $\mathscr E:D(\mathscr E)\times D(\mathscr E)\rightarrow\mathbb R$ is a symmetric form and
satisfies the weak sector condition. In this situation the associated operator with $(\mathscr E, D(\mathscr E))$ is
defined as follow

\begin{equation}
  \begin{array}{ll}
           D(A):=\{u\in D(\mathscr E)\,|\, \exists\varphi\in\mathcal H\text{ s.t. }\mathscr E(u,v)=(\varphi,v)
           \forall v\in D(\mathscr E)\} \\[0.2cm]
             Au:=\varphi.
          \end{array}
\end{equation}

%Let us recall the following interesting proposition saying that the infinite sum of closable symmetric bilinear
%forms is also closable
%\begin{proposition}
 %Let $(\mathscr E^k,D(\mathscr E^k)),\, k\in\mathbb N$, be closable symmetric bilinear forms on $\mathcal H$. Define
 %\[
 % \mathcal D:=\{u\in\cap_{k\in\mathbb N}D(\mathscr E^k)|\sum_{k=1}^{\infty}\mathscr E^k(u,u)<\infty\}
 %\]
%\[
% \mathscr E(u,v):=\sum_{k=1}^{\infty}\mathscr E^k(u,v),\quad u,v\in \mathcal D
%\]

%Then $(\mathscr E,\mathcal D)$ is closable.
%\end{proposition}

Recall that a positive definite bilinear form $(\mathscr
E,D(\mathscr E))$ on $\mathcal H$ is said closable on $\mathcal H$
if for all $u_n,\, n\in\mathbb N$, such that $\mathscr
E(u_n-u_m)\to_{n,m\to\infty} 0$ and $u_n\to 0$ in $\mathcal H$, it
follows that $\mathscr E(u_n)\to 0$.

Now we replace $\mathcal H$ by the concrete Hilbert space
$L^2(E;m):=L^2(E;\mathscr B;m)$ with the usual $L^2-$inner product
where $(E;\mathscr B;m)$ is a measure space. As usual we set for
$u,v: E\rightarrow\mathbb R$, $u\vee v:= \mathrm{sup}(u,v),\,
u\wedge v:=\mathrm{inf}(u,v),\, u^+:=u\vee 0,\, u^-:=-u\wedge 0$,
and we write $f\geq g$ or $f<g$ for $f,g\in L^2(E;m)$ if the
inequality holds $m-$a.e. for corresponding representatives.

A symmetric coercive closed form $(\mathscr E,D(\mathscr E))$ on $L^2(E;m)$ is called a symmetric Dirichlet form
if for all $u\in D(\mathscr E)$, one has that $u^+\wedge 1\in D(\mathscr E)$ and $\mathscr E(u^+\wedge 1)\leq
\mathscr E(u)$.

\begin{definition}
 \begin{enumerate}
  \item[(i)] An increasing sequence $(F_k)_{k\in\mathbb N}$ of closed subsets of $E$ is called $\mathscr E-$nest if
  $\bigcup_{k\geq 0} D(\mathscr E)_{F_k}$ is dense in $D(\mathscr E)$ with respect to $\mathscr E_1^{1/2}$, where
  \[D(\mathscr E)_F:=\{u\in \ D(\mathscr E):u=0\text { in }E\setminus F\}.\]
  \item[(ii)] A set $ N$ is called $\mathscr E-$exceptional if $N\subset\cap_{k\in\mathbb N}F_k^c$ for some
  $\mathscr E-$nest $(F_k)_{k\in\mathbb N}$.
  \item[(iii)] We say that a property of points in $E$ holds $\mathscr E-$quasi-everywhere ( $\mathscr E-$q.e.), if
  the property holds outside some $\mathscr E-$exceptional set.
 \end{enumerate}

\end{definition}

\begin{definition}
 A Dirichlet form $(\mathscr E, D(\mathscr E))$ on $L^2(E;m)$ is called \textit{quasi-regular Dirichlet form} if
 \begin{enumerate}
  \item[(i)] There exists an $\mathscr E-$ nest $(E_k)_{k\in\mathbb N}$ consisting of compact sets.
  \item[(ii)] There exists an $\mathscr E_1-$dense subset of $D(\mathscr E)$ whose elements have
  $\mathscr E-$quasi-continues $m-$versions.
  \item[(iii)] There exists $u_n\in D(\mathscr E),\,n\in\mathbb N$, having $\mathscr E-$quasi-continues $m-$versions
  $\tilde u_n,\, n\in\mathbb N$, and an $\mathscr E$-exceptional set $N\subset E$ such that
  $\{\tilde u_n|n\in\mathbb N\}$ separates the pints of $E\setminus N$.
 \end{enumerate}
\end{definition}
~\\

Now fix a measurable space $(\Omega,\mathscr F)$ and  a filtration $(\mathscr F_t)_{t\in[0,\infty]}$ on $(\Omega,
\mathscr F)$. Let $E$ be a Hausdorff topological space and $\mathscr B(E)$ denotes its Borel $\sigma-$algebra. We
adjoint to $E$ an extra point $\Delta$(cemetery) as an isolated point to obtain a Hausdorff topological space
$E_\Delta=E\cup\{\Delta\}$ with Borel algebra $\mathscr B(E_\Delta):=\mathscr B(E)\cup\{B\cup\{\Delta\}
|B\in\mathscr B(E)\}$. Any function $f:E\rightarrow\mathbb R$ is extended as a function on $E_{\Delta}$ by putting
$f(\Delta)=0$. Given a positive measure $\mu$ on $(E_\Delta,\mathscr B(E_\Delta))$ we define a positive measure
$P_\mu$ on $(\Omega,\mathscr F)$ by
\[
 P_\mu(A):=\int_{E_\Delta}P_z(A)\mu(dz),\, A\in\mathscr F
\]

\begin{definition}
 Let $M=(\Omega,\mathscr F, (X_t)_{t\geq 0},(P_z)_{z\in E_\Delta})$ be a Markov process with state space $E$,
 life time $\xi$ and the corresponding filtration $(\mathscr F_t)$. $M$ is called \textit{right process}
 (w.r.t. $(\mathscr F_t)$) if it has the following additional properties
 \begin{enumerate}
  \item[(A)] (Normal property) $P_z(X_0=z)=1$ for all $z\in E_\Delta$.
  \item[(B)] (Right continuity) For each $\omega\in\Omega$, $t\mapsto X_t(\omega)$ is right continuous on $[0,\infty[$.
  \item[(C)] (Strong Markov property) $(\mathscr F_t)$ is right continuous and every $(\mathscr F_t)-$stopping time
  $\sigma$ and every $\mu\in\mathcal P(E_\Delta)$
  \[
   P_\mu(X_{\sigma+t}\in A|\mathscr F_\sigma)=P_{X_\Delta}(X_t\in A),\, P_\mu-\text{a.s.}
  \]
for all $A\in\mathscr B(E_\Delta),\, t\geq 0$.
 \end{enumerate}
\end{definition}

Now we fix $M$ a right process with state space $E$ and life time $\xi$. $(X_t)_{t\geq 0}$ is measurable then
\[
 p_tf(z):=p_t(z,\varphi):=E_z[\varphi(X_t)],\,z\in E,\, t\geq 0,\, \varphi\in\mathscr B(E)^+
\]
define a submarkovian semigroup of kernels on $(E,\mathscr B(E))$.

Let $(\mathscr E,D(\mathscr E))$ be a Dirichlet form on $L^2(E;m)$
and $(T_t)_{t\geq 0}$ the associated sub-markovian strongly
continuous semigroup on $L^2(E;m)$. A right process $M$ with state
space $E$ and transition semigroup $(p_t)_{t\geq 0}$ is called
associated with $(\mathscr E,D(\mathscr E))$ if $p_tf$ is an
$m-$version of $T_tf$ for all $t>0$. If ,in addition, $p_tf$ is
$\mathscr E-$quasi-continuous for all $t>0$ and $f\in\mathscr
B_b(E)\cap L^2(E;m)$, $M$ is called properly associated with
$(\mathscr E,D(\mathscr E))$.

\begin{theorem}
 Let $E$ be a metrizable Lusin space. Then a Dirichlet form $(\mathscr E,D(\mathscr E))$ on $L^2(E;m)$ is
 quasi-regular if and only if there exists a right process $M$ associated with $(\mathscr E,D(\mathscr E))$. In this
 case $M$ is always properly associated with $(\mathscr E,D(\mathscr E))$.
\end{theorem}

A well known characterization of local regular Dirichlet forms still
valid in the case of quasi-regular Dirichlet forms. Let $E$ be a
Lusin topological space and $(\mathscr E,D(\mathscr E))$ a
quasi-regular Dirichlet form on $L^2(E;m)$. Note that since $E$ is
strongly Lindel\"of, the support of a positive measure on
$(E,\mathscr B(E))$ can be defined as follow: for a $\mathscr
B(E)-$measurable function $u$ on $E$ we set
\begin{equation}\label{eq:supp}
 \mathrm{supp}[u]:=\mathrm{supp}[|u|.m]
\end{equation}and call $\mathrm{supp}[u]$ the support of $u$. It is clear that by \eqref{eq:supp} $\mathrm{supp}[u]$
is well-defined for all $u\in L^2(E;m)$. As usual we say that $(\mathscr E,D(\mathscr E))$ have the local property
(or is local) if $\mathscr E(u,v)=0$ for any functions $u,v\in D(\mathscr E)$ with compact disjoint support.

Let now $M=(\Omega,\mathscr F,(X_t)_{t\geq 0},(P_z)_{z\in
E_\Delta})$ be a right process with state space $E$ and life time
$\xi$ associated with $(\mathscr E,D(\mathscr E))$. Then $(\mathscr
E,D(\mathscr E))$ has the local property if and only if $M$ has
continuous sample paths. More precisely
\[
 P_z(t\mapsto X_t\text{ is continuous on }[0,\xi[)=1,\, \text{ for }\mathscr E-\text{q.e. }z\in E.
\]
In this case, $M$ is said to be a diffusion.

Now we present a general ''local compactification`` method that enables us to associate to a quasi-regular Dirichlet
form on an arbitrary topological space a regular Dirichlet form on a locally compact separable metric space. This is
done in such a way that we can transfer results obtained in the later ''classical`` framework to the more general
situation involving quasi-regular Dirichlet forms.

Let $E$ be a Hausdorff topological space and $(\mathscr E,D(\mathscr
E))$ a quasi-regular Dirihlet form on $L^2(E;m)$. Let $(\hat
E,\hat{\mathscr B})$ be a measurable space and let
$i:E\rightarrow\hat E$ be a $\mathscr B(E)/\hat{\mathscr
B}-$measurable map. Let $\hat m=m\circ i^{-1}$ and define an
isometry $\hat i:L^2(\hat E;\hat m)\rightarrow L^2(E;m)$ by defining
$\hat i(\hat u)$ to be $m-$class represented by $\tilde u\circ i$
for any $\hat{\mathscr B}-$measurable $\hat m-$version $\tilde u\in
L^2(\hat E;\hat m)$. Note that the range of $\hat i$ is always
closed, but, of course, in general strictly smaller than $L^2(E;m)$.
Define
\begin{equation}
  \begin{array}{ll}
            D(\hat{\mathscr E}):=\{\hat u\in L^2(\hat E;\hat m)\, |\, \hat i(\hat u)\in D(\mathscr E)\}\\[0.2cm]
             \hat{\mathscr E}(\hat u,\hat v):=\mathscr E(\hat i(\hat u),\hat i(\hat v)),\quad \hat u,\hat v\in
             D(\hat{\mathscr E}).
          \end{array}
\end{equation}

Then clearly $(\hat{\mathscr E},D(\hat{\mathscr E}))$ is a symmetric
positive definite bilinear form on $L^2(\hat E;\hat m)$ satisfying
the weak sector condition. $(\hat{\mathscr E},D(\hat{\mathscr E}))$
is called the image of  $(\mathscr E,D(\mathscr E))$ under $i$.

By \cite[Theorem VI.1.2]{MR}, there exists an $\mathscr E-$nest $(E_n)_{n\geq 0}$ consisting of compact metrizable
sets in $E$ and locally compact separable metric space $\hat Y$ such that
\begin{enumerate}
 \item[(i)] $\hat Y$ is a local compactification of $Y:=\cup E_n$ in the following sense: $\hat Y$ is a locally
 compact space containing $Y$ as a dense subset and $\mathscr B(\hat Y):=\{A\in\mathscr B(\hat Y)\,|\, A\subset Y\}$.
 \item[(ii)] The trace topologies on $E_k$ induced by $E$, $\hat Y$ respectively, coincides for every $k\in\mathbb N$.
 \item[(iii)] The image  $(\hat{\mathscr E},D(\hat{\mathscr E}))$ of $(\mathscr E,D(\mathscr E))$ under the
 inclusion map $i:Y\subset\hat Y$ is a regular Dirichlet form on $L^2(\hat Y;\hat m)$ where $\hat m:=m\circ i^{-1}$
 is a positive Radon measure on $\hat Y$.
\end{enumerate}

Let now $M=(\Omega,\mathscr F,(X_t)_{t\geq 0},(P_z)_{z\in R_\Delta})$ be a right process properly associated with the
quasi-regular Dirichlet form $(\mathscr E,D(\mathscr E))$ on $L^2(E;m)$. Then there exists an $\mathscr E-$exceptional
set $N\subset E$ such that $E\setminus N$ is $M-$invariant and if $\hat M$ is the trivial extension to $\hat E$ of
$M_{|E\setminus N}$, then $\hat M$ is a Hunt process properly associated with the regular Dirichlet form
$(\hat{\mathscr E},D(\hat{\mathscr E}))$ on $L^2(\hat E;\hat m)$.

One can then transfer all results obtained within the analytic
theory of regular Dirichlet forms on locally compact separable
metric spaces (cf. \cite{F80}) to quasi-regular Dirichlet forms on
arbitrary topological spaces. For example, the one-to-one
correspondence between smooth measures and the positive continuous
additive functionals holds. Moreover the well-known Fukushima
decomposition Theorem holds true also. Recall that a positive
measure $\mu$ is called smooth if it charges no $\mathscr
E$-exceptional set and there exists an $\mathscr E-$nest
$(F_n)_{n\in\mathbb N}$ of compact subsets of $E$ such that
$\mu(F_n)<\infty$ for all $n\in\mathbb N$. The one-to-one
correspondence is given by
\[
 \lim_{t\downarrow 0 }E_m[\frac{1}{t}\int_0^tf(X_s)\,dA_s]=\int f\,d\mu, \text{ for all }f\in\mathscr B^+(E)
\]
where $(A_t)_{t\geq 0}$ is a PCAF's of $M$. Moreover, by
\cite[Theorem VI.2.5]{MR}, or \cite[Theorem 4.3]{AR} we have, for
all $\tilde u$ a $\mathscr E-$quasi-continuous $m-$version of $u$,
the following Fukushima decomposition
\[
 \tilde u(X_s)-\tilde u(X_0)=M_t^{[u]}+N_t^{[u]}
\]
where $M^{[u]}:=(M_t^{[u]})_{t\geq 0}$ is a martingale additive
functional of finite energy and $N^{[u]}:= (N^{[u]})_{t\geq 0}$ is a
continuous additive functional of zero energy.

We will apply Fukushima's decomposition in section
\ref{section:c.sm} to obtain a componentwise semimartingale property
of the infinite dimensional reflecting Brownian motion. As in the
finite dimensional case \cite{BHs}, one need a characterization of
bounded variation of $N^{[u]}$(see Lemma \ref{lem:bv}), which we
prove using the transfer method described above.

\subsection{Abstract Wiener space}

In this article we will deal with measure space $(\mo,\mathscr
B(\mo),\gamma)$, where $\mo$ is an open set of a separable Banach
space $E$ endowed with a centered nondegenerate Gaussian measure
$\gamma$. We recall then some facts about Gaussian measures from
\cite{Bo} in a more general framework of locally convex space. Let
$E$ be a locally convex space, and $E'$ its dual space. We call
cylindrical sets (or cylinders) the sets in $E$ which have the form
\[
 C=\{x\in E\, |\, (l_1(x),\dots,l_n(x))\in C_0\},\, l_k\in E'
\]
where $C_0\in\mathscr B(\mathbb R^n)$ is called a base of $C$ and
denote by $\mathscr E(E)$ the $\sigma-$ field generated by all
cylindrical subsets of $E$. In other words, $\mathscr E(E)$ is the
minimal $\sigma-$ field, with respect to which all continuous linear
functionals on $E$ are measurable. It is clear that $\mathscr E(E)$
is contained in the Borel $\sigma-$field $\mathscr B(E)$, but may
not coincide with it. However, in our forthcoming situation where
$E$ is a separable Banach space, the equality $\mathscr
E(E)=\mathscr B(E)$ holds true. A probability measure $\gamma$
defined on the $\sigma-$field $\mathscr E(E)$, generated by $E'$, is
called Gaussian if, for any $f\in E'$, the induced measure
$\gamma\circ f^{-1}$ on $\mathbb R$ is Gaussian. The measure
$\gamma$ is called centered (or symmetric) if all measures
$\gamma\circ f^{-1},\, f\in E'$ are centered. It is well-known that
a Gaussian measure $\gamma$ is characterized by its mean
$a_{\gamma}(f):(E')^*\rightarrow E'$ defined by $a_{\gamma}(f)= \int
f(x)\gamma(dx)$, and the covariance operator
$R_{\gamma}:E'\rightarrow (E')^*$ defined by $R_{\gamma}(f)(g)= \int
(f(x)-a_{\gamma}(f))(g(x)-a_{\gamma}(g))\gamma(dx)$, where $X^*$
denote the algebraic dual of $X$. Note that, by Fernique Theorem, we
have $E'\subset L^2(\gamma)$.

We consider in what follow only centered Gaussian measures $\gamma$
on $E$(i.e. $a_{\gamma}=0$) and we denote by $E'_{\gamma}$ the
closure of $E'$ embedded in $L^2(\gamma)$, with respect to the norm
of $L^2(\gamma)$. The space $(E'_{\gamma},\|.\|_{L^2(\gamma)})$ is
called the reproducing kernel Hilbert space of the measure $\gamma$.
Put $|h|_{H(\gamma)}:=\mathrm{supp}\{l(h)\, :\, l\in E',\,
\|l\|_{L^2(\gamma)}\leq 1\}$ and $H(\gamma):=\{h\in E\,:\,
|h|_{H(\gamma)}<\infty\}$. The space $H(\gamma)$ is called the
Cameron-Martin space. In the literature it is also called the
reproducing kernel Hilbert space.

Note that one can extend $R_{\gamma}$ from $E'$ to $E'_{\gamma}$,
and by \cite[Lemma 2.4.1]{Bo} the Cameron-Martin space is precisely
the space of elements $h\in E$ such that there exists $g\in
E'_{\gamma}$ with $h=R_{\gamma}(g)$. In this case
$|h|_{H(\gamma)}=\|g\|_{L^2(\gamma)}$ and we say that the element
$g$ (we use the notation $\hat h:=g$) is associated with the vector
$h$ or is generated by $h$. The relation determining $\hat h$ is
$f(h)=\int_E f(x)\hat h(x)\gamma(dx),\, f\in E'$ and the
Cameron-Martin space $H(\gamma)$ is equipped with the inner product
$(h,k)_{H(\gamma)}:=(\hat h,\hat k)_{L^2(\gamma)}$. The
corresponding norm is $|h|_{H(\gamma)}= \|\hat h\|_{L^2(\gamma)}$.

Recall that a (finite nonnegative) measure $\mu$ defined on the
$\sigma-$field $\mathscr B(E)$ is called Radon, if for every
$B\in\mathscr B(E)$ and every $\epsilon>0$, there exists a compact
set $K_{\epsilon}\subset B$ with $\mu(B\setminus
K_{\epsilon})<\epsilon$ and called tight if this condition is
satisfied for $B=E$. For example, in our forthcoming situation of a
separable Banach spaces, all measures on $\mathscr B(E)$ are Radon.
By \cite[Theorem 3.2.7]{Bo}, for a Radon Gaussian measure $\gamma$
on a locally convex space $E$, the Hilbert spaces $E'_{\gamma}$ and
$H(\gamma)$ are separable. Moreover, if $\gamma$ is centered then
$E'_{\gamma}$ has countable orthonormal basis, consisting of
continuous linear functionals $f_n$ \cite[Corollary 3.2.8]{Bo}. Once
more, let $\gamma$ be a centered Radon Gaussian measure on $E$, then
by \cite[Theorem 3.6.1]{Bo} the topological support of $\gamma$ (the
minimal closed set of full measure) coincides with the affine
subspace $\overline{H(\gamma)}$, where the closure is meant in $E$,
in particular the support of $\gamma$ is separable. We say that the
Radon Gaussian measure $\gamma$ is nondegenerate if its topological
support is the whole space. It is clear that a centered Gaussian
measure is nondegenerate precisely when its Cameron-Martin space is
everywhere dense.

A triple $(i,H,B)$ is called an abstract Wiener space if $B$ is a
separable Banach space, $H$ is a separable Hilbert space,
$i:H\rightarrow B$ a continuous linear embedding with dense range,
and the norm $q$ of $B$ is measurable on $H$ ( more precisely
$q\circ i$) in the sense of Gross (see \cite[Definition 3.9.2]{Bo}).
Clearly, when $\gamma$ is a centered nondegenerate Gaussian measure
on a separable Banach space $E$, then $(i,H(\gamma),E)$ is an
abstract Wiener space where $i$ is the natural embedding of
$H(\gamma)$ in $E$.

Now denote by $\mathscr FC^{\infty}$ the collection of all
functions, on a locally convex space $E$, of the form:
$f(x)=\varphi(l_1(x),\dots,l_n(x)),\,\varphi\in C_b^{\infty}(\mathbb
R^n)$, $l_i\in E',\,n\in\mathbb N$. Such functions are called smooth
cylindrical functions. A radon measure $\mu$ on $E$ is called
differentiable along a vector $h\in E$ (in the sense of Formin) if
there exists a function $\beta_h^{\mu}\in L^1(\mu)$ such that, for
all smooth cylindrical functions $f$, the following integration by
parts formula holds true:
\[
 \int_E\partial_hf(x)\mu(dx)=-\int_Xf(x)\beta_h^{\mu}(x)\mu(dx).
\]

where $\partial_h f(x)=\lim_{t\to 0}(f(x+th)-f(x))/t$. The function $\beta_h^{\mu}$ is called logarithmic
derivative of the measure $\mu$ along $h$. By \cite[Proposition 5.1.6]{Bo}, for a Radon Gaussian measure on $E$,
$H(\gamma)$ coincides with the collection of all vectors of differentiability. In addition, if $h\in H(\gamma)$ then
$\beta_h^{\gamma}=-\hat h$. Remark that, in \cite{AKR}, when $E$ is a separable Banach space and $H=H(\gamma)$, the
well admissible elements are exactly the elements of $H(\gamma)$, see also \cite{AR}.

%\marginpar{Gaussian}

 %%%%%%%%%%%%%%%%--NEW SECTION--%%%%%%%%%%%%%%%%%%%%%%%%%

\section{Gaussian Sobolev space}\label{sec:gauss}

In this section we develop the notion of relative Gaussian capacity
associated with Gaussian Sobolev spaces $W^{1,2}(\mo,\gamma)$, where
$\mo$ is an arbitrary open set on a separable Banach space $E$
endowed with a nondegenerate centered Gaussian measure $\gamma$. The
starting point is an idea developed in \cite{CL} to define Sobolev
spaces $W^{1,2}(\mo,\gamma)$ by Lipschitz functions as starting
points, but for open sets of the form $\mo= \{x\in E\,:\, G(x)<0\}$,
where $ G$ is a certain Borel function on $E$. Most results in this
section are developed in \cite{KS}, but because of the paper still
not yet published we announce all results with complete proofs.

\subsection{Gaussian Sobolev space}

Let $E$ be a separable real Banach space and $\gamma$ a
nondegenerate centered Gaussian measure on $\mathscr B(E)$, the
Borel $\sigma-$algebra of $E$. The Cameron-Martin space of $\gamma$
is denoted by $H(\gamma)$, which is continuously and densely
embedded in $E$. We say that a function $\varphi:E\rightarrow\mathbb
R$ is $H-$differentiable at $x$ if there is $v\in H(\gamma)$ such
that $f(x+h)-f(x)=[v,h]_H+\circ(|h|_H)$, for every $h\in H(\gamma)$.
In this case $v$ is unique and we set $D_Hf(x)=v$. Moreover for
every unite vector $l\in H(\gamma)$ the directional derivative
$D_H^lf(x):=\lim_{t\to 0}(f(x+tl)-f(x))/t$ exists and coincides with
$[D_Hf(x),l]_H$. The domain of $D_H$ is the Gaussian Sobolev space
$W^{1,2}(\gamma)$ (see \cite[Section 5.2]{Bo}), defined as the
completion of the smooth cylindrical functions under the norm

 \[
  \|f\|_{W^{1,2}(\gamma)}^2:=\int_E|f(x)|^2d\gamma+\int_E\|D_Hf(x)\|^2d\gamma
 \]

 Now let $\mo$ be an open set of $E$. In \cite{CL}, the Sobolev space $W^{1,2}(\mo,\gamma)$ was defined by using
 Lipschitz functions as starting points for open sets of the form $\mo=\{x\in E:  G(x)<0\}$, where
 $ G$ is a Borel version of an element of $W^{1,2}(\gamma)$. In \cite{KS}, the same approach was reproduced
 but for arbitrary open sets. Let $\varphi\in\mathrm{Lip}(\mo)$ and $\breve{\varphi}$ a Lipschitz continuous extension
 to the whole $E$. Since $\mathrm{Lip}(E)\subset W^{1,2}(\gamma)$ (\cite[Example 5.4.10]{Bo}), $D_H\breve{\varphi}$
 is well defined. Note that when $\tilde{\varphi}$ is another Lipschitz continuous extension of $\varphi$ to the
 whole $E$, then $D_H\breve{\varphi}=D_H\tilde{\varphi}$ $\gamma-$ a.e. by \cite[Lemma 5.7.7]{Bo}. We may thus define
 $\Dho:\mathrm{Lip}(\mo)\longrightarrow L^2(\mo,\gamma;H)$ by setting

 \[
 \Dho\varphi:=D_H\breve{\varphi}_{|\mo}
 \]
 where $\breve\varphi$ is any extension of $\varphi$ to an element of $\mathrm{Lip}(E)$.

\begin{lemma}\label{lem:closable}
The operator $D_H^\mo$ is closable.
\end{lemma}

\begin{proof}
Let a sequence $(\varphi_n) \subset \mathrm{Lip}(\mo)$ be given with $\varphi_n \to 0$ in $L^2(\mo, \gamma)$ and
$D_H^\mo \varphi_n \to \Phi$ in $L^2(\mo, \gamma; H)$. We have to prove that $\Phi =0$. To that end, let
$v \in W^{1,2}(\gamma; H)$ be such that $\mathrm{supp}(v) \subset \mo$. We note that,
by \cite[Theorem 5.8.3]{Bo}, $v$ belongs to the domain of the divergence operator $\delta$. Moreover,
by \cite[Lemma 5.8.10]{Bo} also $\delta (v)$ has support in $\mo$. Consequently,
\begin{align*}
    \int_\mo [\Phi, v]_H\, d\gamma & = \lim_{n\to\infty} \int_E [ D_H\breve \varphi_n, v]_H \, d\gamma\\
        & = - \lim_{n\to\infty} \int_E \breve \varphi_n \delta (v) d\gamma \\
        &= - \lim_{n\to \infty} \int_\mo \varphi_n \delta (v) \, d\gamma = 0.
\end{align*}
where $\breve \varphi_n$ is any extension of $\varphi_n$ to an
element of $\mathrm{Lip}(E)$. Thus, $\int_\mo [\Phi, v]_H\, d\gamma
= 0$ for all $v \in W^{1,2}(\gamma; H)$ with support in $\mo$. Since
such $v$ separate the points in $L^2(\mo, \gamma; H)$, it follows
that $\Phi =0$.
\end{proof}

By slight abuse of notation, we denote the closure of $D_H^\mo$ also by $D_H^\mo$. The domain of
$D_H^\mo$ is denoted by $W^{1,2}(\mo, \gamma)$ which is a Banach space
with respect to the norm
\[
    \|\varphi\|_{W^{1,2}(\mo, \gamma)}^2 := \|\varphi\|_{L^2(\mo,\gamma)}^2 +
    \|D_H^\mo \varphi\|_{L^2(\mo,\gamma; H)}^2.
\]
Note that $W^{1,2}(\mo, \gamma)$ is continuously embedded into $L^2(\mo, \gamma)$.\medskip

It is a consequence of \cite[Theorem 5.11.2]{Bo} that, for a
Lipschitz continuous function $\varphi$, the derivative $D_H
\varphi$ exists $\gamma$-a.e.\ as G\^ateaux derivative. Moreover,
$|D_H \varphi|_H$ is almost surely bounded. This has the following
consequence, which we will use later on.

\begin{lemma}\label{lem:mult}
If $\varphi \in W^{1,2}(\mo, \gamma)$ and $\psi \in \mathrm{Lip}(\mo)$, then $\varphi\psi \in W^{1,2}(\mo, \gamma)$
and
\begin{equation}\label{eq:productrule}
    D_H^\mo (\varphi \psi) = (D_H^\mo \varphi )\psi + \varphi (D_H^\mo \psi).
\end{equation}
Moreover, if $\psi \in \mathrm{Lip}(E)$ with $\psi |_{\mo^c} \equiv 0$, then also
$\varphi\psi 1_{\mo} \in W^{1,2}(\gamma)$.
\end{lemma}
\begin{proof}
If $\mo = E$ and both $\varphi$ and $\psi$ are Lipschitz continuous, then \eqref{eq:productrule} follows
from \cite[Theorem 5.11.2]{Bo} and the product rule for G\^ateaux derivatives. Restricting to $\mo$, we
have \eqref{eq:productrule} for Lipschitz continuous $\varphi$ and $\psi$ and for general $\mo$. The case
of general $\varphi$ follows by approximation, using the closedness of $D_H^\mo$. The addendum also
follows by approximation.
\end{proof}

Immediately from the Lemma \ref{lem:mult} one can prove \textit{the
hypoth\`ese de repr\'esentabilit\'e},

\begin{proposition}
 The Dirichlet form $(\mathscr E_{\mo},W^{1,2}(\mo,\gamma))$ satisfies the `` hypoth\`ese de repr\'esentabilit\'e '',
 i.e.
 \[
  2\mathscr E_{\mo}(\varphi,\varphi\psi)-\mathscr E_{\mo}(\varphi^2,\psi)
  =\int_\mo [\Dho\varphi(z),\Dho\psi(z)]_H\,\gamma(dz)
 \]
 for all $\varphi \in W^{1,2}(\mo, \gamma)$ and $\psi \in \mathrm{Lip}(\mo)$.

\end{proposition}

Let us now address some order properties of $W^{1,2}(\mo, \gamma)$.

\begin{lemma}\label{lem:order}
If $\varphi \in W^{1,2}(\mo, \gamma)$, then also $\varphi^+ \in W^{1,2}(\mo,\gamma)$. Moreover, we have
$D_H^\mo(\varphi^+) =
1_{(0,\infty)}\circ \varphi\cdot D_H^\mo \varphi$.
\end{lemma}

\begin{proof}
Let $f \in C^1(\mathbb R)$ with bounded derivative and $\varphi \in
W^{1,2}(\mo,\gamma)$. We claim that $f\circ \varphi \in
W^{1,2}(\mo,\gamma)$ and $D_H^\mo(f\circ \varphi) = f'\circ \varphi
\cdot D_H^\mo \varphi$. Indeed, by definition, there exists a
sequence $(\varphi_n)\subset \mathrm{Lip} (\mo)$ such that
$\varphi_n \to \varphi$ in $L^2(\mo, \gamma)$ and
$D_H\breve{\varphi}_n|_\mo \to D_H^\mo \varphi$ in $L^2(\mo, \gamma;
H)$. As is well known, see \cite[Remark 5.2.1]{Bo}, $f\circ
\breve{\varphi}_n \in W^{1,2}(\gamma)$ with
$D_H(f\circ\breve{\varphi}_n) = f'\circ \breve{\varphi}_n\cdot
D_H\breve{\varphi}_n$. Using the boundedness and continuity of $f'$,
it is immediate from dominated convergence that
$D_H(f\circ\breve{\varphi}_n)|_\mo \to f'\circ \varphi\cdot D_H^\mo
\varphi$ in $L^2(\mo,\gamma; H)$. The claim thus follows from the
closedness of $D_H^\mo$.\smallskip

Now let $\psi_n(t) = nt 1_{(0,n^{-1})}(t)+ 1_{[n^{-1}, \infty)}(t)$ and
$\phi_n(t) = \int_{-\infty}^t\psi_n(s)\, ds$.
By the above, $\phi_n\circ \varphi \in W^{1,2}(\mo,\gamma)$ with $D_H^\mo(\phi_n\circ \varphi)
= \psi_n\circ \varphi \cdot D_H^\mo \varphi$. As $\phi_n\circ \varphi \to \varphi^+$ in $L^2(\mo,\gamma)$ and
$\psi_n\circ \varphi \cdot D_H^\mo \varphi \to 1_{(0,\infty)}\circ \varphi \cdot D_H^\mo \varphi$ in
$L^2(\mo,\gamma; H)$,
the lemma follows from the closedness of $D_H^\mo$.
\end{proof}

Since $\varphi\wedge \psi = \varphi- (\varphi-\psi)^+$, we immediately obtain the following.
\begin{corollary}\label{cor:maxderivative}
If $\varphi,\psi \in W^{1,2}(\mo,\gamma)$, then $\varphi\wedge \psi \in W^{1,2}(\mo,\gamma)$ and
\[
    D_H^\mo (\varphi\wedge \psi) = 1_{\{\psi\leq \varphi\}}D_H^\mo \psi + 1_{\{\psi>\varphi\}}D_H^\mo \varphi.
\]
\end{corollary}

The bilinear form $\EO : W^{1,2}(\mo,\gamma)\times W^{1,2}(\mo, \gamma)\to\mathbb R$, defined by
\begin{equation}\label{eq:form}
\EO (\varphi,\psi) = \int_\mo [D_H^\mo \varphi, D_H^\mo \psi]_H\, d\gamma
\end{equation}
is densely defined, symmetric, positive semidefinite and closed. It follows immediately from Corollary
\ref{cor:maxderivative} that $\varphi\wedge 1 \in W^{1,2}(\mo,\gamma)$ whenever $\varphi \in W^{1,2}(\mo, \gamma)$
and, in this case, $D_H^\mo (\varphi\wedge 1) = 1_{ \{\varphi \leq 1\}} D_H^\mo \varphi$. Thus
\[
    \EO (\varphi\wedge 1) = \int_{\{\varphi \leq 1\} }\|D_H^\mo \varphi\|_H^2\, d\gamma \leq \int_\mo
    \|D_H^\mo \varphi\|_H^2 \, d\gamma
         = \EO (\varphi).
\]

Consequently, $\EO$ is a Dirichlet form on $L^2(\mo,\gamma)$.% in the sense of \cite[Definition I.3.3.1]{BH}.

 %%%%%%%%%%%%%%%%%%%%%%%%%%%%%%%%%%%%%%%%%%%%%%%%%%%%%%%%%%

\subsection{Gaussian relative capacity}

Associated with the Dirichlet form $\EO$ is a capacity $\Cap$, see \cite[Section I.8]{BH}. In this article,
we will consider this capacity as a \emph{relative capacity} in the sense of \cite{AW}, i.e.\ we allow to compute
capacities of subsets of $\obar$. To do so, we formally have to consider
$\EO$ as a form on $L^2(X, \mathscr B(X), m)$ instead of $L^2(\mo, \mathscr B(\mo), \gamma|_\mo)$,
where $X = \obar$ and $m(A) = \gamma (A \cap \mo)$ for $A \in \mathscr B(X)$. The definition is as follows.

\begin{definition}\label{def:relcap}
Let $A \subset \obar$. Then the \emph{relative Gaussian capacity} $\Cap (A)$ of $A$ is defined as
\begin{equation}\label{eq:relcap}
    \Cap (A) := \inf \big\{ \|u\|_{W^{1,2}(\mo, \gamma)}^2 :
        \text{$\exists\, U \subset E$ open, s.t.\ $u\geq 1$ $\gamma$-a.e.\ on $U\cap \mo$}\big\} .
\end{equation}
\end{definition}

Standard properties of $\Cap$ are easily verified and follow from the general theory,
see \cite[Proposition I.8.1.3]{BH}.
\begin{proposition}\label{prop:capprop}
Let $\mo \subset E$ be open. Then the following statements hold.
\begin{enumerate}
\item $\gamma (A)  \leq \Cap (A)$ for all $A \subset \obar$ such that $A \in \mathscr B(E)$.
\item For $A,B \subset \obar$ one has
\[
\Cap (A\cup B) + \Cap (A\cap B) \leq \Cap (A) + \Cap (B).
\]
\item For every increasing sequence $(A_n)$ of subsets of $\obar$ one has
\[
    \Cap (A_n) \uparrow  \Cap \Big( \bigcup_{k=1}^\infty A_k \Big).
\]
\item For every decreasing sequence $(K_n)$ of compact subsets of $\obar$ one has
\[
    \Cap(K_n)\downarrow \Cap\Bigl(\bigcap_{k=1}^\infty K_k\Bigr).
\]
\item For every sequence $(A_n)$ of subsets of $\obar$ one has
\[
    \Cap  \Big( \bigcup_{k=1}^\infty A_k \Big) \leq \sum_{k=1}^\infty \Cap(A_k).
\]
\end{enumerate}
\end{proposition}

The following is now a consequence of Choquet's capacity theorem
\cite[Corollary~30.2]{Ch}.
\begin{proposition}\label{prop:cho-cap}
Let $\mo\subset E$ be open and $A\subset\obar$. If $A\in\mathscr B(E)$, then
\[
    \Cap(A) = \sup\{\Cap(K) : K\subset A \text{ compact}\}.
\]
\end{proposition}

For $\mo = E$ we write $\Ecap$ rather than $\Ecap_E$ and refer to $\Ecap$ as \emph{Gaussian capacity}. This
Gaussian capacity has been extensively studied in the literature, see, e.g., \cite[Section II.3]{BH}. In view
of \cite[Theorem 5.7.2]{Bo}, it follows that the capacity $C_{2,1}$, considered in \cite[Section 5.9]{Bo} is
equivalent with $\Ecap$, in the sense that for certain constants $\alpha,\beta > 0$, we have \[\alpha C_{2,1}(A)
\leq  \Ecap (A) \leq \beta C_{2,1}(A)\] for all $A \subset E$.

We adopt the following terminology from \cite{AW}.
\begin{definition}
\begin{enumerate}
\item A subset $A$ of $\obar$ is called \emph{relatively polar} if $\Cap (A) = 0$.

\item Some property is said to hold
on $\obar$ \emph{relatively quasi everywhere} (r.q.e.) if it holds outside a relatively polar set.
\end{enumerate}
\end{definition}

We now compare relatively polar sets with polar sets, i.e.\ sets $A$ with $\Ecap (A) = 0$. It turns out that
polar subsets of $\obar$ are relatively polar. The converse is true for subsets of $\mo$.

\begin{proposition}\label{prop:relpol} Let $A \subset \obar$ and $B \subset \mo$.
\begin{enumerate}
\item $\Cap (A) \leq \Ecap (A)$. In particular, polar sets are relatively polar.
\item $\Cap (B) = 0$ if and only if $\Ecap (B) = 0$.
\end{enumerate}
\end{proposition}

\begin{proof}

(1) It follows from the density of $\mathrm{Lip}(E)$ in $W^{1,2}(\gamma)$, that
$\varphi_{|\mo}\in W^{1,2}(\mo,\gamma)$ for every $\varphi \in W^{1,2} (\gamma)$. Thus (1) is immediate from
the definition.

(2) We only need to prove that $\Cap (B) = 0$ implies $\Ecap (B) = 0$. Let
$F_n := \{ x \in \mo : d(x, \mo^c) \geq n^{-1}\}$ for all $n\in\mathbb N$. Then $F_n$ is closed, contained
in $\mo$ and $F_n \uparrow \mo$. Thus $B \cap F_n \uparrow B$. It suffices to show that $\Ecap (B\cap F_n) = 0$
because then, by Proposition \ref{prop:capprop}~(3), $\Ecap (B) = \lim_n \Ecap (B\cap F_n) = 0$. So let
$n\in\mathbb N$ be fixed. Then there exists a Lipschitz function $\varphi$ with $1_{F_n} \leq \varphi \leq 1_{\mo}$.
Since $\Cap (B) = 0$, there exists a sequence $(f_k)$ in $W^{1,2}(\mo, \gamma)$ and open sets $U_k \subset E$
containing $B$ with $f_k \geq 1$ $\gamma$-a.e.\ on $U_k\cap \mo$ and $\|f_k\|_{W^{1,2}(\mo, \gamma)}^2 \to 0$.
As a consequence of Lemma \ref{lem:mult}, $g_k := \varphi f_k \in W^{1,2}(\gamma)$ and $\|g_k\|_{W^{1,2}(\gamma)}
\leq c\|f_k\|_{W^{1,2}(\mo, \gamma)}$ for a certain constant $c$. It follows that $\Ecap (B\cap F_n) = 0$, which
finishes the proof.

\end{proof}

As a consequence of part (1), the relative capacity $\Cap$ inherits tightness from the Gaussian capacity $\Ecap$.

\begin{corollary}\label{cor:tight}
The relative capacity $\Cap$ is tight, i.e.\ for every $\epsilon > 0$, there exists a compact set
$K_\epsilon \subset \obar$ such that
\[
    \Cap (\obar\setminus K_\epsilon ) < \epsilon.
\]
\end{corollary}

\begin{proof}
The Gaussian capacity $\Ecap$ is tight, see \cite[Theorem 5.9.9]{Bo} (cf.\ also \cite[Proposition II.3.2.4]{BH}).
Consequently, given $\epsilon >0$, there exists a compact set $\hat{K}_\epsilon \subset E$ with
$\Ecap (E \setminus \hat{K}_\epsilon) \leq \epsilon$. The set $K_\epsilon := \obar\cap \hat{K}_\epsilon$ is compact
and, by Proposition \ref{prop:relpol}(1)
\[
    \Cap (\obar\setminus K_\epsilon) \leq \Ecap (\obar\setminus K_\epsilon) \leq
    \Ecap (E \setminus \hat{K}_\epsilon) \leq \epsilon. \qedhere
\]
\end{proof}

 It now follows that the form $\EO$ is a quasi-regular Dirichlet form on $L^2(\mo,\gamma)$.
 Thus there exists a right processus $\mathbf M=(\Omega,\mathscr F,(X_t)_{t\geq 0},(P_z)_{z\in E_{\Delta}})$ with
 state space $\obar$ and life time $\xi$, which is properly associated with $\EO$. Moreover, one can prove, with
 the same method as in \cite[Example 1.12 (1)]{MR}, that

 \begin{proposition}
 The quasi-regular Dirichlet form $\EO$ is local.
 \end{proposition}

 \begin{proof} To prove the locality it is sufficient to show that

 \begin{equation}
 \Dho\varphi=0\text{ }\gamma-\text{a.e. on } \obar\setminus{\mathrm{supp}[\varphi]}\text{ for all }\varphi
 \in W^{1,2}(\mo,\gamma)
 \end{equation}

 To this aim we use the following identity \eqref{eq:productrule}
  \begin{equation}\label{eq:id}
 \Dho(\varphi\psi)=\psi\Dho\varphi+\varphi\Dho\psi \quad\varphi\in W^{1,2}(\mo,\gamma),\psi\in\mathrm{Lip}(\mo)
 \end{equation}

 Let $\varphi\in W^{1,2}(\mo,\gamma)$. By \cite[Proposition V.4.17]{MR} there exists a Lipschitz function $\psi$
 such that $0\leq\psi\leq 1_{\obar\setminus\mathrm{supp}[\varphi]}$ and $\psi>0$ r.q.e. on $\obar\setminus
 \mathrm{supp}[\varphi]$. Hence by the identity \eqref{eq:id}
 \[
 0=\psi\Dho\varphi+\varphi\Dho\psi
 \]

 and thus

 \[
 \psi\Dho\varphi=\varphi\Dho\psi=0
 \]

 Consequently $\Dho\varphi=0$ $\gamma-$a.e. on $\obar\setminus\mathrm{supp}[\varphi]$.

 \end{proof}

 As a consequence of the locality of $\EO$, the associated right process $\mathbf M$ is in fact
 a diffusion process (Strong Markov process with continuous sample paths).

 \subsection{Quasi-continuous representatives}

 We next establish the existence of certain representatives of elements of $W^{1,2}(\mo, \gamma)$ that are
 unique up to a relatively polar set. This allows to consider pointwise properties of elements which
 hold r.q.e.\ instead of merely $\gamma$-a.e. For example, we will see that using these representatives a
 convenient description of the closed lattice ideals of $W^{1,2}(\mo, \gamma)$ can be given.

\begin{definition}
A function $\varphi\colon\obar\to\mathbb R$ is called \emph{relatively quasi continuous} if for all $\epsilon>0$
there exists an open set $U$ in $E$ such that $\Cap(U\cap\obar)<\epsilon$ and $\varphi$ restricted to
$\obar\setminus U$ is continuous. Moreover, a subset $M\subset\obar$ is called \emph{relatively quasi open} if for all $\epsilon>0$
there exists an open set $U$ in $E$ such that $\Cap(U\cap\obar)<\epsilon$ and $M\cup U$ is open in $E$.
\end{definition}

The following proposition provides us with relatively quasi continuous representatives and collects two basic
properties that allow to lift pointwise properties from $\gamma$-a.e. to r.q.e. It suffices to note that in our
setting property (D) of~\cite[Section~I.8.2]{BH} holds. So the proposition is a consequence of
\cite[Propositions I.8.1.6 and I.8.2.1]{BH}. For the corresponding properties in the case $\mo=E$, see
also \cite[Lemma 5.9.5 and Theorem 5.9.6]{Bo}.

\begin{proposition}\label{prop:rqc-rep}
For every $\varphi\in W^{1,2}(\mo, \gamma)$ there exists a
relatively quasi continuous and measurable representative
$\tilde{\varphi}\colon\obar\to\mathbb R$, which is unique up to
equality r.q.e. Moreover, one has the following.
\begin{enumerate}
\item
Let $\varphi\in W^{1,2}(\mo, \gamma)$. Then $\varphi\geq 0$ $\gamma$-a.e.\ if and only if $\tilde{\varphi}\geq 0$
r.q.e.
\item
If $\varphi_n\to \varphi$ in $W^{1,2}(\mo, \gamma)$, then after going to a subsequence one may assume
$\tilde{\varphi}_n\to\tilde{\varphi}$ r.q.e.
\end{enumerate}
\end{proposition}

 %%%%%%%%%%%%%%%%%%%-NEW SECTION-%%%%%%%%%%%%%%%%%%%%%%%

 \section{Hausdorff-Gauss measures}

 In each tentative to establish a Skorohod representation one remark that establishing an integration by parts formula is
 a fundamental first step. In a new article \cite{CL} such integration by parts was proved for open sets with some
 non restrictive regularity. Before to give the integration by parts we will define the well known Hausdorff-Gauss
 measure of Feyel-de La Pradelle. It is the equivalent notion of Hausdorff measures in the infinite dimensional
 spaces. We first introduce such a measures and then we give the integration by parts result. Our reference in this
 section  will be always the paper \cite{CL}. We follow then \cite[Subsection 2.1]{CL} and we recall that $E$ is a
 separable Banach space endowed with a nondegenerate centered Gaussian measure $\gamma$ and $H$ is the relevant
 Cameron-Martin space.

 We recall first of all the definitions of the $1-$codimensional Hausdorff-Gauss measures that will be
 considered in the sequel.

 If $m\geq 2$, and $F=\mathbb R^m$ is equipped with a norm $|.|$, we define
 \[
  \theta^F(dx):=\frac{1}{(2\pi)^{m/2}}\exp(-|x|^2/2)H^{m-1}(dx)
 \]
$H^{m-1}$ being the spherical $m-1$ dimensional Hausdorff measure in $\mathbb R^m$, namely
\[
 H^{m-1}(A):=\lim\inf_{\delta\mapsto 0}\{\sum_{i\in\mathbb N}\omega_{m-1}r^{m-1}_i\,:\,
 A\subset\bigcup_{i\in\mathbb N}B(x_i,r_i) ,r_i<\delta\,\forall i  \}
\]
where $\omega_{m-1}$ is the Lebesgue measure of the unite sphere in $\mathbb R^{m-1}$.

 For every finite dimensional subspace $F\subset E$ we consider the orthogonal (along $H$) projection on $F$:
 \[
  x\mapsto\sum_{i=1}^{m}\langle x,f_i\rangle_H f_i,\quad x\in H
 \]
where $\{f_i\, :\, i=1,\dots,m\}$ is any orthogonal basis of $F$.
Then there exists a $\gamma-$ measurable projection $\pi^F$ on $F$,
defined in the whole $E$, that extends it. Its existence is a
consequence of \cite[Theorem 2.10.11]{Bo}, which states that for
every $i$ there exists a unique (up to changes on sets with
vanishing measure) linear and $\mu-$measurable function
$l_i:X\rightarrow\mathbb R$ that coincides with $x\mapsto\langle
x,f_i\rangle_H$ on $H$. Then we set
\[
 \pi^F(x):=\sum_{i=1}^ml_i(x)f_i.
\]

If $f_i\in Q(E'),\, f_i=Q(\hat f_i)$ for some $\hat f_i\in E'$, then
$\langle x,f_i\rangle=\hat f_i(x)$ for every $x\in H$ and the
extension is obvious, $l_i(x)=\hat f_i(x)$ for every $x\in E$. In
particular if $E$ is a Hilbert space, it is convenient to choose an
orthonormal basis $\{e_k:\, k\in \mathbb N\}$ of $E$ made by
eigenvectors of $Q$. If $Qe_k=\lambda_ke_k$, the function $l_i$ is
the $L^2(E,\gamma)$ limit of the sequences of cylindrical functions
\[
 l^m_i(x):=\sum_{k=1}^m\frac{\langle x,e_k\rangle_E\langle f_i,e_k\rangle_E}{\lambda_k},\quad m\in\mathbb N
\]
which is noted $W_{Q^{-1/2}f_i}$ in \cite{Da}. If $F$ is spanned by
finite number of elements of the basis $\mathcal
V=\{v_k:=\sqrt{\lambda_k}e_k:\, k\in \mathbb N\}$ of $H$, say
$F=\mathrm{span}\{v_1,\dots,v_m\}$, then
\[
 \pi^F(x)=\sum_{i=1}^m\langle x,Q^{-1}v_i\rangle_Ev_i=\sum_{i=1}^m\langle x,e_i\rangle_Ee_i,
\]
namely $\pi^F$ coincides with the orthogonal projection in $E$ over the subspace spanned by $e_1$, $\dots$,$e_m$.

Let $\tilde F$ be the kernel of $\pi^F$. We denote by $\gamma^F$ the image measure of $\gamma$ of $F$ through
$\pi_F$, and by $\gamma^F$ the image measure of $\gamma$ on $\tilde F$ through $I-\pi^F$. We identify in a standard
way $F$ with $\mathbb R^m$, namely the element $\sum_{i=1}^mx_if_i\in F$ is identified with the vector
$(x_1,\,\dots,x_m)\in\mathbb R^m$ and we consider the measure $\theta^F$ on $F$.

We stress that the norm and the associated distance used in the
definition of $\theta^F$ are inherited from the $H-$norm on $F$, not
from the $E-$norm. For instance, if $E=\mathbb R^m=F$, then
$dH^{m-1}=dS\circ Q^{-1/2}$ where $dS$ is the usual
$(m-1)-$dimensional spherical Hausdorff measure. So, for every Borel
set $E$,
\[
 \theta ^F(A)=\frac{1}{(2\pi)^{m/2}}\int_{Q^{-1/2}(A)}e^{-|y|^2/2}dS.
\]

In the general case, for any Borel (or, more general, Suslin) set
$A\subset E$ we set
\[
 \rho^F:=\int_{\tilde F}\theta^F(A_x)d\gamma_F(x),
\]
where $A_x:=\{y\in F:\, x+y\in A\}$. By \cite[Proposition 3.2]{Fe},
the map $F\mapsto\rho^F(A)$ is well defined (namely, the function
$x\mapsto\theta^F(A_x)$ is measurable with respect to $\gamma_F$)
and increasing, i.e. if $F_1\subset F_2$ then
$\rho^{F_1}\leq\rho^{F_2}$. This is sketched in \cite{Fe}, a
detailed proof is in \cite[Lemma 3.1]{AMP}. By the way, this is the
reason to choose the spherical Hausdorff measure in $\mathbb R^m$:
if the spherical haussdorf measure is replaced by the usual
Hausdorff measure, such a monotonicity condition may fails.

The Hausdorff-Gauss measure of Feyel-de La Pradelle is defined by
\begin{equation}
 \rho(A):=\mathrm{sup}\{\rho^F(A):\, F\subset H, \text{ finite dimensional subspace}\}
\end{equation}

Similar definition were considered in \cite{AMP}
\begin{equation}
 \rho_1(A):=\mathrm{sup}\{\rho^F(A):\, F\subset Q(E'), \text{ finite dimensional subspace}\}
\end{equation}

and under the assumption that $\mathcal V\subset Q(E')$ in \cite{Hi}, the following Hausdorff-Gaussian measure was
defined
\begin{equation}
 \rho_{\mathcal V}:=\mathrm{sup}\{\rho^F(A):\, F\subset H, \text{ spanned by a finite number of elements of }
 \mathcal V\}
\end{equation} where $\rho_{\mathcal V}$ could depend on the choice of the basis $\mathcal V$.

The three type of Hausdorff-Gaussian measures can be compared  as
follow
\[
 \rho(A)\geq \rho_1(A)
\]
and  when $\mathcal V\subset Q(E')$ we have
\[
\rho_1(A)\geq \rho_{\mathcal V}(A).
\]

 The following Proposition is important in the sense that it permits to us to say that $\rho$ is a smooth measure and
 then to associate with it a positive continuous additive functional $L^{\rho}_t$ which we call, as in the finite
 dimensional case, the local time of $\mathbf M$ corresponding to $\rho$ with the Revuz correspondence. One can find
 the proof in \cite[Theorem 9]{FL}
 \begin{proposition}
 The Hausdorff-Gauss measure of Feyel-de La Pradelle $\rho$ charges no set of zero relative Gaussian capacity.
 \end{proposition}

 Now we give the integration by parts under the following not restrictive assumptions,

 \begin{Hypo}\label{Hypo}
 ~

 \begin{enumerate}
  \item[(A.1)] $ G\in W^{2,q}(E,\gamma)$ for each $q>1$,
  \item[(A.2)] $\gamma( G^{-1}(-\infty,0))>0,\text{ } G^{-1}(0)\neq \emptyset$,
  \item[(A.3)] there exist $\delta>0$ such that $1/|\Dho G|_H\in L^q( G^{-1}(-\delta,\delta),\gamma)$
  for each $q>1$.
 \end{enumerate}
 \end{Hypo}

 The following theorem (see \cite[Corollary 4.2]{CL}) give a definition of a trace operator from a limiting procedure of a sequence of Lipschitz
 functions,
 \begin{theorem}\label{thm:Tr}

 For each $p>1$ and $\varphi\in W^{1,p}(\mo,\gamma)$ there exists $\psi\in\bigcap_{q<p} L^p(\{ G=0\},\rho)$
 with the
 following property: if $(\varphi_n)_n\subset\mathrm{Lip}(E)$ are such that $({\varphi_n}_{|\mo})$ converge to $\varphi$
 in $W^{1,p}(\mo,\gamma)$, the sequence $({\varphi_n}_{|\mo})$ converges to $\psi$ in $ L^q(\{ G=0\},\rho)$,
 for every $q<p$. In addition, if the condition
 \begin{equation}\label{eq:esssup}
 \gamma-\mathrm{ess}\text{ }\sup_{x\in\mo} \mathrm{div}\left(\frac{\Dho G}{|\Dho G|_H}\right)<\infty
 \end{equation}
 holds then ${\varphi_n}_{|\{ G=0\}}$ converges in $L^p(\{ G=0\},\rho)$.

 \end{theorem}
 Theorem \ref{thm:Tr} justify the following definition (see \cite[Definition 4.3]{CL})

 \begin{definition}\label{def:Tr}
 For each $\varphi\in W^{1,p}(\mo,\gamma)\text{ }p>1$, we define the trace $\mathrm{Tr}\varphi$ of $\varphi$ at
 $\{ G=0\}$ as the function $\psi$ given by Theorem \ref{thm:Tr}.

 \end{definition}
 Let $\{v_k\,|\,k\in\mathbb N\}$ be an orthonormal basis of $H(\gamma)$. Now the integration by parts of functions
 in $W^{1,2}(\mo,\gamma)$ is as follow (see \cite[Corollary 4.4]{CL})

 \begin{theorem}\label{thm:ibp}
 For every $\varphi\in W^{1,2}(\mo,\gamma)$, we have

 \begin{equation}
 \int_{\mo}D_k^{\mo}\varphi d\gamma=\int_{\mo}\hat{v}_k\varphi d\gamma+
 \int_{\partial\mo}\frac{D_k^{\mo} G}{|\Dho G|_H}\mathrm{Tr}\varphi d\rho
 \end{equation}
 where $\mathrm{Tr}$ is the operator trace as defined in Definition \ref{def:Tr}.
 \end{theorem}

 \begin{proposition}
 For every $\varphi\in W^{1,p}(E,\gamma)$, the trace of $\varphi_{|\mo}$ at $ G^{-1}(0)$ coicides
 $\rho-$a.e. with the restriction to $ G^{-1}(0)$ of any continous version $\tilde{\varphi}$ of $\varphi$.
 \end{proposition}

 \section{Componentwise Skorohod decomposition}\label{section:c.sm} To obtain the Skorohod decomposition we use,
 as in the finite dimensional situation, the well known Fukushima decomposition theorem which holds in the situation
 of quasi-regular Dirichlet forms by using the transfer method see \cite[Theorem 4.3]{AR},\cite[Theorem VI.3.5]{MR}.

 \begin{theorem}\label{thm:fuku}

 Let $\varphi\in W^{1,2}(\mo,\gamma)$ and let $\tilde{\varphi}$ be a relatively quasi-continous
 $\gamma-$version of $\varphi$. Then the additive functional $(\tilde{\varphi}(X_t)-\tilde{\varphi}(X_0))_{t\geq 0}$
 of $\mathbf M$ can uniquely be represented as

  \[
  \tilde{\varphi}(X_t)-\tilde{\varphi}(X_0)=M_t^{[\varphi]}+N_t^{[\varphi]}, t\geq 0
  \]
 where $M^{[\varphi]}:=(M_t^{[\varphi]})_{t\geq 0}$ is a MAF of $\mathbf M$ of finite energy and
 $N^{[\varphi]}:=(N_t^{[\varphi]})_{t\geq 0}$ is a CAF of $\mathbf M$ of zero energy.
\end{theorem}

To evaluate the bracket $\langle M^{[\varphi]}\rangle$ of the
martingale additive functional $M^{[\varphi]}$ for $\varphi\in
W^{1,2}(\mo,\gamma)$ we use a standard technic as for the finite
dimensional case \cite{BHs} and used in the infinite dimensional
framework in \cite[Proposition 4.5]{AR} in the case $\mo=E$ with
help of the transfer method. The proof still the same in our
framework. Remark that one need no regularity assumption on $\mo$
and then in this step the open set $\mo$ still arbitrary.

\begin{proposition}\label{pro:M}
Let $\varphi\in W^{1,2}(\mo,\gamma)$, then
\begin{equation}
\langle
M^{[\varphi]}\rangle_t=\int_0^t[\Dho\varphi(X_s),\Dho\varphi(X_s)]_Hds,\quad
t\geq 0
\end{equation}

\end{proposition}
\begin{proof}Recall that we are always considering $\mathscr E_\mo$ as a form on $L^2(\overline\mo,m)$ as done in
Section \ref{sec:gauss}. Endowing $\overline\mo$ with the topology induced by the separable Banach space $E$,
$\overline\mo$ is a Polish space. We define now the function $\theta$ as follow,
\begin{equation}
\theta(z):=
\begin{cases}
 [\Dho\varphi(z),\Dho\varphi(z)]_H & \text{ if } z\in\obar\\
 0                                 &  \text{ if } z\in\hat{\obar}\setminus\obar
\end{cases}
\end{equation}
and $\hat{N}_t:=\int_0^t\theta(\hat{X}_s)ds,t\geq 0$. Then it follows by \cite[Lemma 5.1.6 and Theorem 3.2.3]{F80}
that
\[
\hat{P}_z[\hat{N}_t<\infty,t\geq 0]=1
\]
for $\widehat{\text{r.q.e.}}\text{ } z\in\hat{\obar}$. Consequently, $(\hat{N}_t)_{t\geq 0}$ is a CAF of
$\hat{\mathbf M}$ and we have for $f:\hat{\obar}\rightarrow[0,\infty[,\mathscr B(\hat{\obar})-$measurable, that
\begin{equation}
\begin{split}
\frac{1}{t}\int_{\hat{\obar}}\hat{E}_z\left[\int_0^tf(\hat{X}_s)d\hat{N}_s\right]d\hat{\gamma}&=
                                   \frac{1}{t}\int_0^t\int_{\hat{\obar}}\hat{p}_s(f\theta)d\hat{\gamma}ds\\
                                   &=\frac{1}{t}\int_0^t\int_{\hat{\obar}}f\theta\hat{p}_s1d\hat{\gamma}ds\\
                                   &=\int_{\hat{\obar}}f\theta d\hat{\gamma}
\end{split}
\end{equation}
where the last step follows by the fact that $(X_t)_{t\geq 0}$ is markovian and then so is $(\hat X_t)_{t\geq 0}$
thus $\hat{p}_s1=1 \, \hat{\gamma}-$a.e. By \cite[Theorem 5.1.3]{F80} it follows that the unique
smooth measure that is associated to $\hat{N}:=(\hat{N}_t)_{t\geq 0}$ is $\theta\hat{\gamma}$. For
$\varphi\in D(\hat{\mathscr E}_{\mathcal O})$ let $\hat{\gamma}_{\langle\varphi\rangle}$ denote the unique
smooth measure associated with $\langle\hat{M}^{[\varphi]}\rangle$. We want to show also that
\[
\hat{\gamma}_{\langle\varphi\rangle}=\theta\hat{\gamma}\
\]
By \cite[Theorem 5.2.3]{F80} we know that if $\varphi_n:=\sup(\inf(\varphi,n),-n),\text{ }n\in\mathbb N$,
then for all $f\in D(\hat{\mathscr E}_{\mathcal O})\cap L^{\infty}(\obar,m)$
 \[
 2\hat{\mathscr E}_{\mathcal O}(\varphi_n.f,u_n)-\hat{\mathscr E}_{\mathcal O}(\varphi_n^2,f)=\int_{\mo}f(z)
 [\Dho\varphi_n(z),\Dho\varphi_n(z)]_Hd\gamma
 \]
 Consequently, by \cite[Theorem 5.2.3]{F80}

 \begin{equation}\label{eq:varphi}
 \hat{\gamma}_{\langle\varphi_n\rangle}(dz)=[\Dho\varphi_n(z),\Dho\varphi_n(z)]_H\hat{\gamma}(dz)
 \end{equation}
Since by \cite[Proof of Lemma 5.4.6]{F80}
 \[
 \left(\left(\int |f|d\hat{\gamma}_{\langle\varphi\rangle}\right)^{\frac{1}{2}}-
 \left(\int |f|d\hat{\gamma}_{\langle\varphi_n\rangle}\right)^{\frac{1}{2}}\right)^2\leq
 2\|f\|_{\infty}\hat{\mathscr E}_{\mathcal O}(\varphi-\varphi_n,\varphi-\varphi_n),
 \]
\eqref{eq:varphi} implies that $\hat{\gamma}_{\langle
\varphi\rangle}=\theta\hat{\gamma}$. By uniqueness part of
\cite[Theorem 5.1.3]{F80} we now have that
$\langle\hat{M}^{[\varphi]}\rangle=\hat{N}$, hence clearly
\[
\langle M^{[\varphi]}\rangle_t=\int_0^t[\Dho\varphi(X_s),\Dho\varphi(X_s)]_Hds, t\geq 0
\]
and the Theorem is proven.
\end{proof}

\begin{remark}
Here we denote with  $\widehat{.}$  what is denoted in \cite[Chapter V]{F80} by $.^{\sharp}$)
\end{remark}

Now we focus on the CAF of zero energy $N^{[\varphi]}$ for
$\varphi\in W^{1,2}(\mo,\gamma)$. Here one cannot use the same
procedure as for the case $\mo=E$ in \cite{AR}. To evaluate
$N^{[\varphi]}$ we shall characterize, as in the regular Dirichlet
forms framework, the boundedness of its variation which is an easy
task by using the transfer method (see Lemma \ref{lem:bv}).

An additive functional (AF) $A$ is then said to be of bounded
variation, if $A_t(\omega)$ is of bounded variation in $t$ on each
compact subinterval of $[0,\xi(\omega)[$ for every fixed $\omega$ in
a defining set of $A$, i.e. its total variation process
\[
|N|_t(\omega)=\sup\sum_{i=0}^{n-1}\|N_{t_i}(\omega)-N_{t_{i-1}}(\omega)\|_E
\]
is finite, where the supremum is taken over all finite partitions
$0=t_0<t_1<\dots<t_n=t<\xi(\omega)$.

Let $(\mathscr E,D(\mathscr E))$ a quasi-regular Dirichlet form on
$L^2(X,m)$ where $X$ is some Luzin space and $m$ a full support
measure on $X$. We have then the following Lemma,

\begin{lemma}\label{lem:bv}
The following two conditions are equivalent to each other for $\varphi\in D(\mathscr E)$

\begin{enumerate}
 \item  $N^{[\varphi]}$ is a CAF of bounded variation,
 \item there exist smooth measures $\nu^1$ and $\nu^2$ such that
 \begin{equation}\label{eq:bvthm}
 \mathscr E(u,v)=\langle\nu_k,\tilde v\rangle,\quad\forall v\in D(\mathscr E)_k
 \end{equation} for every $k$. Here $\nu_k$ is the restriction to $F_k$ of the difference $\nu^1-\nu^2$. $\{F_k\}$
 being the common nest associated with $\nu^1$ and $\nu^2$. $D(\mathscr E)_k$ is the space defined by
 \[
 D(\mathscr E)_k:=\{\varphi\in D(\mathscr E):\tilde{\varphi}=0 \text{ q.e. on }E\setminus F_k\}
 \]
\end{enumerate}

\end{lemma}

\begin{proof}
By \cite[Theorem V. 1.6]{F80} we may extend $M$ on $E$ to a Hunt
process $\hat M$ on $\hat E$. Every PCAF $(A_t)_{t\geq 0}$ can be
extended (e.g. by zero) to a PCAF $(\hat A_t)_{t\geq 0}$ of $\hat M$
and vis versa. Let $\varphi\in D(\mathscr E)$, we denote by
$\hat{\varphi}$ the extension by zero on $\hat E\setminus E $ of
$\varphi$, and we suppose that $N^{[\varphi]}$ is of bounded
variation, thus $\hat N^{[\hat{\varphi}]}$ is also of bounded
variation, then by \cite[Theorem 5.3.2]{F80} there exist smooth
measures $\hat\nu^1$ and $\hat\nu^2$ such that
\[
\hat{\mathscr E}(\hat{\varphi},\hat{\psi})=\langle\hat{\nu}_k,\widetilde{\hat{\psi}}\rangle,\quad\forall\hat{\psi}
\in D(\hat{\mathscr E})_k
\]for all $k$ and where $\hat{\nu}_k$ is the restriction to $\hat F_k$ of the difference $\hat{\nu}^1-\hat{\nu}^2$.
$\{\hat F_k\}_k$ being the common nest associated with $\hat{\nu}^1$ and $\hat{\nu}^2$ and
\[
D(\hat{\mathscr E})_k:=\{\hat{\varphi}\in D(\hat{\mathscr E}):\widetilde{\hat{\varphi}}=0
\text{ }\widehat{\text{q.e.}}\text{ on }\hat E\setminus{\hat F_k}\}
 \]

It suffice now to choose $\nu^1=\hat{\nu}^1_{|\mathscr B(E)}$ and $\nu^2=\hat{\nu}^2_{|\mathscr B(E)}$ and by
\cite[Theorem 1.2, Corollary 1.4 and Proposition 1.5 p:174-176]{MR} one can come back to \eqref{eq:bvthm}. The
converse follows with the same transfer technic.

\end{proof}

We want now to give a componentwise Skorohod decomposition, but a
technical problem arise since the indexation on the derivatives is
on $H(\gamma)$ but the one of the component process $(\langle
k,X_t\rangle)_{t\geq 0}$ of the $E-$valued process $(X_t)_{t\geq 0}$
are on $E'$. This problem can easily be surrounded by the following
procedure: First of all recall that $H(\gamma)\hookrightarrow E$
continuously and densely. By identifying $H(\gamma)$ and
$H(\gamma)'$ we have that
\[
E'\hookrightarrow H(\gamma)\hookrightarrow E
\]
continuously and densely in both embeddings. Let $j_H:E'\rightarrow
H(\gamma)$ to be the left embedding. Thus for all $l\in E'$, the
functional $h\rightarrow ~_{E'}\langle l,h\rangle_E$ is continuous
in $H(\gamma)$. Hence there exists a unique $j_H(l)\in H(\gamma)$
such that
\begin{equation}\label{eq:jh}
 ~_{E'}\langle l,h\rangle_E=[j_H(l),h]_H
\end{equation}

Note that as $H(\gamma)=R_{\gamma}(E'_{\gamma})$, one can write $j_H$ explicitly as follow: $j_H(l)=R_{\gamma}(l)$ for all
$l\in E'$. Since $\gamma$ is centered,  $E'_{\gamma}$ has a countable
orthonormal basis, consisting of continous linear functionals $l_k,\,k\in\mathbb N$ \cite[Corollary 3.2.8]{Bo}.
Let $K=\mathrm{span}\{l_k\in E'\,:\,k\in\mathbb N\}\subset E'$ thus
$\{h_k:=j_H(l_k)\,:\,k\in\mathbb N\}$ forms an orthonormal basis of $H(\gamma)$ (eventually after applying
Gram-Schmidt orthogonalisation). Note that, by Hahn-Banach theorem, $E'$ separates the points of $E$, and
since $K$ is dense in $E'$, then $K$ also separates the points of $E$.

 Now after what is done before, one can always identify $E'\times H(\gamma)$ with $H(\gamma)\times H(\gamma)$ with
 help of the map $j_H$ defined by \eqref{eq:jh}, which means that one can consider the dualisation
 $~_{E'}\langle,\rangle_E$ to coincide with $[,]_H$ when restricted to $E'\times H(\gamma)$. In this situation
 one have a countable subset $K_0=\{l_k,\,k\in\mathbb N\}$ of $E'$ forming an orhonormal basis of $H(\gamma)$
 and separating the points of $E$. Moreover the linear span $K\subset E'$ of $K_0$ is dense in $H(\gamma)$. In this
 and the following sections we fix $K$ and the orthonormal basis $K_0$ of $H(\gamma)$ defined as above.

Now to establish the componentwise Skorohod representation we need
to use the integration by parts in Theorem \eqref{thm:ibp}. We
consider then, in what follow, open sets of the form $\mathscr
O=\{x\in E\,|\, G(x)<0\}$ where $G$ satisfies assumptions
\ref{Hypo}. We define the following coordinate functions: For $l\in
K$, with $|l|_H=1$, define
\[
\varphi_l(z)=~_{E'}<l,z>_E, z\in E
\]

The functions $\varphi_l$ are continuous Lipschitz functions on the
whole $E$, thus the functions ${\varphi_l}_{|\mo}$ are Lipschitz
continuous functions on $\mo$ and belong to $W^{1,2}(\mo,\gamma)$.

\begin{theorem}\label{thm:skorohod}
In the case where $\varphi=\varphi_l$, the Fukushima decomposition of $M^{[\varphi]},\,\varphi\in W^{1,2}(\mo,\gamma)$ in
Theorem \ref{thm:fuku} becomes as follow:
\begin{equation}
\varphi_l(X_t)-\varphi_l(X_0)=W^l_t+\int_0^t\hat
l(X_s)ds+\int_0^t\nu^l_{ G}(X_s)dL_s^{\rho}
\end{equation}
where for all $z\in\obar\setminus S_l$ for some relative polar set
$S_l\subset\obar$ the continous martingale $(W_t^l,\mathscr F_t,
P_z)_{t\geq 0}$ is a one dimensional Brownian motion starting at
zero, $\hat l$ is the element generated by $l$,
\[\nu^l_G=\frac{D_H^lG}{|D_H G|_H}\]
plays the role of the outward normal vector field in the direction
of $l$ and $L^{\rho}_t$ is the positive continuous additive
functional associated with the Gaussian-Hausdorff measure $\rho$ by
Revuz correspondence. Moreover, $L_t^{\rho}$ verify
\begin{equation}\label{eq:L}
 \int_0^t1_{\partial\mo}(X_s)\,dL_s^{\rho}=L_t^{\rho}.
\end{equation}

\end{theorem}

 \begin{proof}
 By Lemma \ref{lem:bv} the AF $N^{[\varphi_l]}$ is of bounded variation and its associated measure
 $\gamma^{[\varphi_l]}$ is uniquely characterized by the equation
 \[
 \int_{\mo}[\Dho\varphi_l,\Dho\psi]_Hd\gamma=\int_{\obar}\psi d\gamma^{[\varphi_l]}
 \]
 for a relatively quasi-continuous function $\psi\in W^{1,2}(\mo,\gamma)$. By the integration by part formula in Lemma
 \ref{thm:ibp} we have

 \begin{equation}
 \begin{split}
\int_{\obar}\psi d\gamma^{[\varphi_l]} &=\int_{\mo}[\Dho\varphi_l,\Dho\psi]_H d\gamma\\
                                       &=\int_{\mo}[l,\Dho\psi]_H d\gamma\\
                                       &=\int_{\mo}D_l^{\mo}\psi d\gamma\\
                                       &=\int_{\mo}\hat{l}\psi d\gamma+\int_{\partial\mo}
                                        \frac{D^l_HG}{|D_H G|_H}\psi d\rho
  \end{split}
 \end{equation}
 which allows us to identify the measure $\gamma^{[\varphi_l]}$ associated to $N^{[\varphi_l]}$, i.e.

 \[
 \gamma^{[\varphi_l]}(dz)=\hat{l}(z)\gamma(dz)+n^l_G(z)\rho(dz)
 \]
 where $\rho$ is the Hausdorff-Gauss measure and
 \[\nu^l_G=\frac{D_H^lG}{|D_H G|_H}\]
 plays the role of the outward normal vector field in the direction of $l$. Consequently, the CAF of zero energy
 $N^{[\varphi_l]}$ must be
 \[
 N^{[\varphi_l]}=\int_0^t\hat{l}(X_s)ds+\int_0^t \nu_G^l(X_s)d L_s^{\rho}
 \]
 where $L^{\rho}_t=(L_t^{\rho})_{t\geq 0}$ is the continous additive functional associated with $\rho$ by the Revuz
 correspondence and by \cite[Theorem 5.1.3, p. 129]{F80} the equality \eqref{eq:L} holds.

 By Proposition \ref{pro:M} we know that
 \begin{equation}\begin{split}
 \langle M^{[\varphi_l]}\rangle_t & =\int_0^t[\Dho\varphi_l(X_s),\Dho\varphi_l(X_s)]_Hds\\
                                  & =\int_0^t|l|_H ^2ds\\
                                  & =t
 \end{split}\end{equation}

 It follows by P. Levy's characterization of Brownian motion that $(M^{[\varphi_l]})_{t\geq 0}$ is an
 $(\mathscr F_t)_{t\geq 0}-$Brownian motion starting at zero under each $P_z$, $z\in\mo\setminus S_l$.

 \end{proof}

 Let $\{l_k,k\in\mathbb N\}$ the orthonormal basis of $H(\gamma)$ as defined above, then it is easy to see that, by
 Theorem \ref{thm:skorohod}, we have solved a certain system of stochastic differential equations. This is announced by
 the following Theorem,
 \begin{theorem}

 The stochastic process
 $(\{~_{E'}\langle l_k,X_t\rangle_E|k\in\mathbb N\},\mathscr F_t,P_z)$ solves, for r.q.e. $z\in\obar$,
 the following system of stochastic differential equations

 \begin{equation}\begin{cases}
dY_t^k&=dW_t^k+\hat l_k(Y_t^k)dt+n_G^k(Y_t^k)dL_t\\
     Y_0^k &=\langle k,z\rangle_{E',E}
 \end{cases},\quad k\in\mathbb N\end{equation}
 where $\{(W_t^k)_{\geq 0},k\in\mathbb N\}$ is a collection of independent one dimensional
 $(\mathscr F_t)_{t\geq 0}-$ Brownian motion starting at zero.
 \end{theorem}

 \begin{proof}
 The result follows from Theorem \ref{thm:skorohod}, and the P. Levy's theorem. In fact, in virtue of the linearity
 of the map $\varphi\mapsto M^{[\varphi]}$ (cf. \cite[Corollary 1, p.139]{F80}) and Proposition \ref{pro:M} one can
 conclude that

 \[
 \langle W_t^k,W_t^{k'}\rangle_t=t[ l_k,l_{k'}]_H=t\delta_{k,k'},\quad t\geq 0
 \text{ and }k,k'\in\mathbb N
 \]
 which means that any vector process $\bar W = \{W_t^1,\dots,W_t^d\}$ is a $d-$dimensional
 $(\mathscr F_t)_{t\geq 0}-$Brownian motion starting at zero under $P_z$ for r.q.e. $z\in\obar$.

 \end{proof}
 %%%%%%%%%%%%%%%%--NEW SECTION--%%%%%%%%%%%%%%%%%%%%%%%%%

 \section{Skorohod decomposition}

 In the last section we had established the Skorohod decomposition for the components  $(X^k_t)_{t\geq 0}
 (k\in\mathbb N)$. Now we are interested in the Skorohod decomposition of the process $(X_t)_{t\geq 0}$. One remarks
 that passing from $(X^k_t)_{t\geq 0}$ to $(X_t)_{t\geq 0}$ is not trivial. In fact, a problem occur when one wants to find an
 $E-$valued Brownian motion $(W_t)_{t\geq 0}$ verifying $~_{E'}\langle l_k,W_t\rangle_E=W_t^k$ and a map
 $\hat l:E\rightarrow E$ such that $~_{E'}\langle l_k,\hat l\rangle_E=\hat l_k$. To do this we mainly follow the
 procedure developed in \cite[Section 6]{AR}. The procedure is based on the crucial technical lemma
 \cite[Lemma 6.1]{AR} that we present also here without proof and we refer to the above cited article for
 detailed one.

 Recall that $E$ is a separable Banach space and denote by $\|.\|_{E'}$ the operator norm on $E'$, we know
 then, by the Banach/Alaoglu-theorem, that
 \[
  B_n':=\{l\in E'|\,\|l\|_{E'}\leq n\},\, n\geq 0,
 \]
equipped with the weak$^*-$topology is compact . Moreover, it is metrizable by some metric $d_n$, hence
in particular separable. Let $D_n\subset K$ be a countable dense subset of $(B_n',d_n),~n\in\mathbb N$, such that
$D_n\subset D_{n+1}$ for every $n\in\mathbb N$. Let $\tilde{D}_n$ be the $\mathbb Q-$ linear span of $D_n$ and set
\begin{equation}\label{eq:Ddef}
 D:=\bigcup_{n\in\mathbb N}\tilde{D}_n
\end{equation}

\begin{lemma}\label{lem:D}
 Let $(\Omega,\mathscr A)$ be an arbitrary measurable space and let $D$ to be as in \eqref{eq:Ddef}. Now let
 $\alpha_l:\Omega\to\mathbb R,\, l\in D$, be $\mathscr A-$ measurable maps. Then there exists an
 $\mathscr A/\mathscr B(E)-$measurable map $\alpha:\Omega\to E$ such that
 \begin{equation}
  ~_{E'}\langle l,\alpha\rangle_E=\alpha_l\quad\text{ for all }l\in D
 \end{equation}
$P-$a.s. for every probability measure $P$ on $(\Omega,\mathscr A)$ satisfaying the following two conditions:
\begin{enumerate}
 \item[(i)] $l\to\alpha_l$ is $\mathbb Q-$linear $P-$a.s.
 \item[(ii)] There exists a probability measure $\nu_P$ on $(E,\mathscr B(E))$ such that
 \begin{equation}\label{eq:lem}
  \int\exp(i\alpha_l)d\, P=\int\exp(i~_{E'}\langle l,z\rangle_E)\nu_P(d\, z)\quad \text{ for all }l\in D
 \end{equation}

\end{enumerate}

\end{lemma}

Lemma \ref{lem:D} will be applied to construct an $E-$valued Wiener
process from the components $W_t^k$, but before let us make some
remarks.

\begin{remark}\label{rem:rem1}~

\begin{enumerate}
 \item[(a)] First of all, let us remark that the evaluation of the martingale part in the Fukushima decomposition
 is not 'disturbed' by whether we work on $E$ or on an open set $\mo$ of $E$. It is why the treatment of the
 martingale part is similar to the one in $E$ as we deal, in the both situations, with $E-$valued Wiener processes
 without any kind of reflection or perturbation. One can see it clealy from the componentewise process, where in
 both situation the martingal part give arise to a one dimensional Brownian motion.

 \item[(b)] One can say the same as in $(a)$ about $\hat l_k$, where $\{l_k\,:\, k\in\mathbb N\}$ is the orthonormal
 basis of the Cameron-Martin space $H(\gamma)$, defined in the last section and $\hat l_k$ is the element generated
 by $l_k$.

 \item[(c)] Note that if $W_t$ is a standard Wiener process in $\mathbb R^n$, then
for any unit vector $v\in\mathbb R^n$, the process $(v, W_t)$ is one dimensional Wiener. Hence
one might try to define a Wiener process in a separable Hilbert space $H$ as a
continuous process $W_t$ with values in H such that, for every unit vector $v\in H$,
the real process $(v,W_t)_H$ is Wiener. However, such a process does not exist if $H$ is infinite dimensional
(see section 7.2 in \cite{Bo}).
\end{enumerate}
\end{remark}

To get around the difficulty apearing in Remark \ref{rem:rem1} (c), let $j_H$ be as defined in \eqref{eq:jh} and
define
\begin{definition}
A continuous random
process $(W_t)_{t\geq 0}$ on $(\Omega,\mathscr F,P)$ with values in $E$ is called a Wiener process
associated with $H$ if, for every $l\in E'$ with $|j_H(l)|_H = 1$, the one dimensional process
$~_{E'}\langle l,W_t\rangle_E$ is Wiener.
\end{definition}

\begin{definition}
Let $\mathscr F_t,\, t > 0$, be an increasing family of $\sigma-$fields. A Wiener process
$(W_t)_{t\geq 0}$ is called an $(\mathscr F_t)_{t\geq 0}-$Wiener process if, for all $t, s\geq\tau$, the random
vector $W_t-W_s$, is independent of $W_\tau$, and the random vector $W_t$ is $\mathscr F_t-$measurable.
\end{definition}

In a more general framework where $E$ is a locally convex space, it follows by \cite[Proposition 7.2.2]{Bo}, that a
Wiener process exists precisly when there exists a Hilbert space $H$ continously and densely embedded into $E$. In
particular in our sitation where $E$ is a separable Banach space and $H(\gamma)$ is the relevent Cameron-Martin
space, then there exists by \cite[Proposition 7.2.3]{Bo} a Wiener process $(W_t)_{t\geq 0}$ associated with
$H(\gamma)$ such that the distribution of $W_1$ coincides with $\gamma$.

Here also and by the identification in the last section, the
definition of the $E-$valued Wiener (or Brownian motion) process can
be reformulated as follow: A continuous random process $(W_t)_{t\geq
0}$ on $(\Omega,\mathscr F,P)$ with values in $E$ is called a Wiener
process (or Brownian motion ) associated with $H(\gamma)$ if, for
every $l\in K$ with $|l|_H = 1$, the one dimensional process
$~_{E'}\langle l,W_t\rangle_E$ is Wiener.

Now remark that, in general, one can not apply Lemma \ref{lem:D} directly to the one dimensional Brownian motion
$W_t^k$ because of the duality product in \eqref{eq:lem}, which justify an extension assumption on the standard
Gaussian cylinder measure on $H(\gamma)$. More precisely, for $t>0$ let $\gamma_t$ denote the standard Gaussian
cylinder measure on $H(\gamma)$, then one have
\[
 \int_{H(\gamma)}\exp(i\langle h,k\rangle_H)\gamma_t(dk)=\exp(-\frac{1}{2}t|h|^2_H),\, h\in H(\gamma)
\]
and each $\gamma_t$ induces a finitely additive measure $\widetilde\gamma_t$ on the cylinder sets of $E$ defined by
\begin{equation}\label{eq:ext}
 \widetilde\gamma_t(A^E_{l_1,\dots,l_n}):=\gamma_t(A^H_{l_1,\dots,l_n})
\end{equation}
where $A^E_{l_1,\dots,l_n}:=\{z\in E\,|\, (~_{E'}\langle l_1,z\rangle_E,\dots,~_{E'}\langle l_n,z\rangle_E)\in A\}$
and $A^H_{l_1,\dots,l_n}:=\{h\in H(\gamma)\,|\, (\langle l_1,h\rangle_H,\dots,\langle l_n,h\rangle_H)\in A\},\,
l_1,\dots,l_n\in E', A\in\mathscr B(\mathbb R^n)$.

In \cite{AR}, the following essential assumption was considered,
\[
\text{Each } \widetilde\gamma_t,\, t>0, \text{(as in \eqref{eq:ext}) extends to a probability measure }
\gamma_t^* \text{ on } (E,\mathscr B(E)).
\]

In our situation we don't need such assumption, since the extension
exists always and it is unique, see Theorem 4.1 in \cite{Ku} and the
paragraph after its proof.

Now, before to apply Lemma \ref{lem:D} as in \cite[Theorem 6.2]{AR} to obtain an $E-$valued Brownian motion from the
componentwise one dimensional Brownian motions $W_t^k$ appearing in Theorem \ref{thm:skorohod}, let us first recall
this important result from \cite[Proposition 1]{Ho}, see also \cite[Theorem 5.1]{Ro}, which permits us to be sure of
the existence of a continous sample paths version of the process constructed by Lemma \ref{lem:D}.

\begin{lemma}\label{lem:ext}
 Let $(Y_t)_{t\in\mathbb R}$ be a mean zero Gaussian stochastic process on a probability space $(\Omega,\mathscr A,P)$
 taking values in a real separable Banach space $(X,\|.\|_X)$. Assume that
 \[
   \lim_{t\to s}E_P[\|Y_s-Y_t\|^2_X]=0,\quad\text{for each }t\in\mathbb R
 \]

 Let $f:\mathbb R^+\rightarrow\mathbb R^+$ be a continous, increasing function such that $f(0)=0$ and that
 \[
  \sup\{E_P[\|Y_s-Y_t\|_X^2]^{1/2}:\, s,t\in\mathbb R,\, |s-t|\leq r\}\leq f(r)
 \]

 Assume that
 \[
  \int_0^1\left(\mathrm{ln}\frac{2}{r}\right)^{1/2}\,df(r)
 \]

 Then for any $n\in\mathbb N$ there exists a constant $\theta_n>0$ and an $\mathscr A-$measurable
function $B_n:\Omega\rightarrow\mathbb R^+$ such that for all $s,t\in[ -n, n]$
\begin{equation}
 \|Y_s(\omega)-Y_t(\omega)\|_X\leq\theta_n\int_0^{2|s-t|}\left(\mathrm{ln}\frac{B_n(\omega)}{r^2}\right)^{1/2}\,df(r),
 \text{ for }P-\text{a.e. }\omega\in\Omega
\end{equation}
In particular, there exists a version $(\widetilde Y_t)_{t\in\mathbb R}$ of $(Y_t)_{t\in\mathbb R}$
(i.e. for each $t\in\mathbb R,\,Y_t=\widetilde Y_t\,P-$a.s. ) which has continous sample paths.

\end{lemma}

\begin{theorem}\label{thm:W}
 There exists a map $W:\Omega\to C([0,\infty[,E)$ having the following properties:
 \begin{enumerate}
  \item[(i)] $\omega\to W_t(\omega):=W(\omega)(t),\, \omega\in\Omega$, is $\mathscr F_t/\mathscr B(E)-$ measurable
  for $t\geq 0$.

  \item[(ii)] There exists a relatively polar set $S\subset E$ such that under each $P_z,\, z\in E\setminus S,\,
  W=(W_t)_{t\geq 0}$ is an $(\mathscr F_t)_{t\geq 0}-$Brownian motion on $E$ starting at $0\in E$ with covariance
  $[,]_H$

  \item[(iii)] For each $k\in\mathbb N,\, ~_{E'}\langle l_k,W_t\rangle=W_t^k,\, t\geq 0,\, P_z-$a.s. for all $z\in E$
  outside a relatively polar set (depending on $k$).
 \end{enumerate}

\end{theorem}
\begin{proof}
 Let $D\subset K$ be as \eqref{eq:Ddef}. Since the maps $l\mapsto \varphi_l$ and $u\mapsto M^{[u]}$ are linear then
 $l\mapsto W_t^l:=(W_t^l)_{t\geq 0}$ is $\mathbb Q-$linear on $D,\, P_z-$a.s. for each $z\in E\setminus S$ and some
 relatively polar set $S$. Consequently $(i)$ in Lemma \ref{lem:D} is satisfied. Moreover, by Theorem
 \ref{thm:skorohod}, $E_z[\exp(iW_t^l)]=\exp(-\frac{1}{2}t|l|^2_{H(\gamma)})$ for all (unite vector)
 $l\in D,\, t\geq 0$. Since $\tilde\gamma_t$ extends to a probability measure $\gamma_t^*$, then $(ii)$ in
 Lemma \ref{lem:D} is also satisfied.
 Now fixing $t\geq 0$ and applying Lemma \ref{lem:D} with $\mathscr A=\mathscr F_t$ and $\alpha_l:=W_t^l$, we obtain
 that there exists an $\mathscr F_t/\mathscr B(E)-$measurable map $\widetilde W_t:\Omega\rightarrow E$ such that
 \begin{equation}\label{eq:Wk}
  ~_{E'}\langle l,\widetilde W_t\rangle_E=W_t^l\quad,\text{ for all (unite vector) }l\in D,\,
  P_z-\text{a.s. for each } z\in E\setminus S.
 \end{equation}
 Remark that the law of $\widetilde W_1$ is precilsy $\gamma_1^*$ and then, by scaling, one obtain that the law
 of $\widetilde W_t-\widetilde W_s$ is $(t-s)\gamma_1^*$(see also \cite[Remark 6.3]{AR}), hence for
 $z\in E\setminus S,\, t,\, s\geq 0$,
 \[
  E_z\left[\|\widetilde W_t-\widetilde W_s\|_E^2  \right]=(t-s)\int\|z\|_E^2\gamma_1^*(dz)
 \]
which is finite by Fernique/Skorohod theorem (cf. \cite[Theorem 3.41]{Str}). Now we apply Lemma \ref{lem:ext} to
$Y_t=\widetilde W_t$ and $f(r)=a.r$ where $a=\int\|z\|_E^2\gamma_1^ *(dz)$, since the independence of the
random variable $B_n$ on $P_z$ can be choosen uniformly for all $P_z,\, z\in E\setminus S$. It then follows that
there exists a version $(W_t)_{t\geq 0}$ of $(\widetilde W_t)_{t\geq 0}$ which is of continous sample paths such
that for each $t\geq 0,\, \omega\mapsto W_t(\omega):=W(\omega)(t),\, \omega\in\Omega,$ is $\mathscr F_t-$measurable
and $W_t=\widetilde W_t,\, P_z$-a.s. for all $z\in E\setminus S$. Since $\mathscr F_t$ is complete, (i) is proven.

By the continuity of the sample paths and \eqref{eq:Wk} it follows that
\[
  ~_{E'}\langle l, W_t\rangle_E=W_t^l\quad,\text{ for all }t\geq 0,\, l\in D,\, P_z-\text{a.s. for each }
  z\in E\setminus S.
\]
which holds also for $l\in K$ by \cite[Corollary 1 (ii), p.139]{F80}. This implies (iii).

It remains to show that $W=(W_t)_{t\geq 0}$ is an $(\mathscr{F}_t)_{t\geq 0}-$Brownian motion on $E$. By Theorem
\ref{thm:skorohod} we may assume that for each unite vector $l\in D$, $(W_t^l,\mathscr F_t, P_z)_{t\geq 0}$ is an
$(\mathscr F_t)_{t\geq 0}-$Brownian motion on $\mathbb R$ for all $z\in E\setminus S$. Hence by \eqref{eq:Wk} under each $P_z$,
$z\in E\setminus S$ , the random variable $~_{E'}\langle l,W_t-W_s\rangle_E$ is mean zero Gaussian with covariance
$(t-s)|l|_H^2=(t-s)$ for all $0 \leq s < t$ and a unite vector $l\in D$. Consequently the same is true for all $l\in E'$. Since
for $0 \leq s < t$ the $\sigma-$algebra $\{(W_t-W_s)^{-1}(B)\,|\, B\in\mathscr B(E)\}$ on $\Omega$ is equal to
the $\sigma-$algebra generated by $\{~_{E'}\langle l,W_t-W_s\rangle_E\,|\, l\in D\}$ on $\Omega$, it follows again
by Theorem \ref{thm:skorohod} and \eqref{eq:Wk} that $W_t-W_s$ is independent of $\mathscr F_s$. Since
$W=(W_t)_{t\geq 0}$ has continuous sample paths and because of part (i), it follows that $W$ is an
$(\mathscr F_t)_{t\geq 0}-$Brownian motion on $E$.

\end{proof}

The following Theorem is a direct consequence of Theorem \ref{thm:skorohod} and Theorem \ref{thm:W}.
\begin{theorem}\label{thm:N}
 There exists a map $N:\Omega\rightarrow C([0,\infty[,\obar)$
 having the following properties
 \begin{enumerate}
  \item[(i)] $\omega\mapsto N_t(\omega):=N(\omega)(t),\,\omega\in\Omega$, is $\mathscr F_t/\mathscr B(\obar)-$ measurable
  for each $t\geq 0$.
  \item[(ii)] For each unite vector $l\in K$, we have
  \[
  ~_{E'}\langle l,N_t\rangle_E=\int_0^t\hat l(X_s)\,ds+\int_0^t\nu_G^l(X_s)d\,L^{\rho}_s,\,  (t\geq 0)
  \]
  $P_z-$a.s. for all $z\in\obar$ outside
  a relatively polar set (depending on $l$).
  \item[(iii)] $X_t=z+W_t+N_t,\, t\geq 0,\, P_z-$a.s. for all $z\in\obar\setminus S$, where $W$ and $S$ are as in
  Theorem \ref{thm:W}.
 \end{enumerate}
\end{theorem}
\begin{proof}
 Define $N:=X-W$ where $X:\Omega\rightarrow C([0,\infty[,\obar)$ is given by
 $X(\omega)(t):=X_t(\omega)-X_0(\omega),\, \omega\in\Omega,\, t\geq 0$. Then (i) holds by Theorem \ref{thm:W} from
 which (ii) and (iii) also follow in virtue of Theorem \ref{thm:skorohod}.
\end{proof}
\begin{theorem}
 There exists a map $W:\Omega\rightarrow C([0,\infty[,E)$ such that
 for r.q.e. $z\in\obar$ under $P_z$, $W=(W_t)_{t\geq 0}$ is an $(\mathscr F_t)_{t\geq 0}-$Brownian motion on $E$
 starting at zero with covariance $[,]_H$ such that for r.q.e. $z\in\obar$
 \begin{equation}\label{eq:skorohod2}
  X_t=z+W_t+\int_0^tX_s\,ds+\int_0^t\nu_G(X_s)\,dL^{\rho}_s
 \end{equation}
where $L^{\rho}_t:=(L^{\rho}_t)_{t\geq 0}$ is a positive continous additive functional which is associated with
$\rho$ by the Revuz correspondence and verify the equality \eqref{eq:L}. In addition
$\nu_G$ is a unite vector defined by
\[
 \nu_G:=\frac{D_H G}{|D_H G|_H}
\]

\end{theorem}
\begin{proof}
 Let $N_t$ be as defined in Theorem \ref{thm:N} and let $\{l_k\,|\,k\in\mathbb N\}$ be the
 orthonormal basis of $H(\gamma)$, as fixed in the last section. Then by Theorem \ref{thm:N}, we have that for all
 $z\in E\setminus S$
 \[
  ~_{E'}\langle l_k,N_t\rangle_E=\int_0^t\hat l_k(X_s)ds+\int_0^t\nu_G^k(X_s)\,dL_s^{\rho}, t\geq 0,\, P_z-\text{a.s.}
 \]
where $\hat l_k$ is the element generated by $l_k$ and $\nu_G^k:=\nu_G^{l_k}$. As $D_H^kG=[l_k,D_H G]_H$ then there
exists $\nu_G$ such that $[l_k,\nu_G]_H=\nu_G^k$, which is given explicitly by
\begin{equation}
 \begin{split}
 \nu_G &=\sum_{k=1}^{\infty}\nu_G^kl_k\\[0.2cm]
     &=\sum_{k=1}^{\infty}\frac{[l_k,D_H G]_H}{|D_H G|_H}l_k\\[0.2cm]
     &=\frac{D_H G}{|D_H G|_H}
  \end{split}
 \end{equation}
Now by \cite[Proposition 5.1.6]{Bo} and \cite[Example 7.3.3 (i)]{Bo} there exists a map $\hat l:E\rightarrow E$ such
that $~_{E'}\langle l_k,\hat{l}\rangle_E=\hat l_k$, which is exactly the identity, i.e. $\hat l(x)=x$ (see also
\cite[Remark 6.8 (ii)]{AR}). Consequently the map $N:\Omega\rightarrow C([0,\infty[,\obar)$ defined in Theorem
\ref{thm:N} is given, for each $z\in E\setminus S$, by
\[
 N_t=\int_0^tX_s\,ds+\int_0^t\nu_G(X_s)\,dL^{\rho}_s,\quad t\geq 0,\, P_z-\text{a.s.}
\]
Define $W_t:=X_t-X_0-N_t,\,t\geq 0$. It follows by Theorem \ref{thm:skorohod} that
\[
 ~_{E'}\langle l, W_t\rangle_E=W_t^{l},\quad t\geq 0,\, l\in D
\]
$P_z-$a.s. for all $z\in E\setminus S$, where $D$ is as in Lemma \ref{lem:D}. It now follows as in the last part of
the proof of Theorem \ref{thm:W} that $W=(W_t)_{t\geq 0}$ is an $(\mathscr F_t)_{t\geq 0}-$Brownian motion on $E$
with covariance $[,]_H$.
\end{proof}

 %%%%%%%%%%%%%%%%--NEW SECTION--%%%%%%%%%%%%%%%%%%%%%%%%%
 \section{Examples}
 We give some examples to illustrate the skorohod representation in infinite dimensions. It includes regions below graphs and Balls.
 %\subsection{Halfspaces}

 \subsection{Regions below graphs}
 We fix $\hat h\in E'$ such that $\|\hat h\|_{L^2(E\gamma)}=1$ and we set $h:=Q(\hat h)$. Then $|h|_H=1$ and
 $\hat h(h)=1$. we split $E=\mathrm{span}\, h\oplus Y$, where $Y=(I-\Pi_h),\, \Pi_h(x)=\hat h(x)h$. The Gaussian
 measure $\gamma\circ(I-\Pi_h)^{-1}$ on $Y$ is denoted by $\gamma_Y$.

 Let $F\in\bigcap_{p>1} W^{2,p}(Y,\gamma_Y)$. Choose any Borel precise version of $F$ ( for example we can choose $F$
 to be a Lipschitz function) and set
 \[
  G:E\mapsto\mathbb R,\quad G(x)=\hat h(x)-F\left((I-\Pi_h)(x)\right).
 \]

 Then, $G\in\bigcap_{p>1} W^{2,p}(E,\gamma)$ and $\Dho G(x)=h-D^{\mo}_{H_Y}F\left((I-\Pi_h)(x)\right)$, so that
 \[
 |\Dho G(x)|_H^2=1+|D_H^{\mo} F(I-\Pi_h)(x)|_{H_Y}^2\geq 1
 \]

 Hence $G$ satisfies assumption \ref{Hypo}. The sublevel $\mo=G^{-1}(-\infty,0)$ is just the region below the
 graph of $F$. The Skorohod decomposition of the infinite dimensional reflecting Ornstein-Uhlenbeck process is
 \begin{equation*}
  X_t=z+W_t+\int_0^tX_s\,ds+\int_0^tn_G(X_s)\,dL_s^{\rho}
 \end{equation*}
 where in this situation $\nu_G$ is defined as follow
 \[
  \nu_G(x)=\frac{ h-D^{\mo}_{H_Y}F\left((I-\Pi_h)(x)\right) }{ (1+|D_H^{\mo} F(I-\Pi_h)(x)|_{H_Y}^2)^{\frac{1}{2}} }
 \]

 \subsection{Balls} In the context of balls we take $E$ to be a separable Hilbert space endowed with a
 nondegenerate centered Gaussian measure $\gamma$, with covariance $Q$. we fix an orthonormal basis
 $\{e_k:k\in\mathbb N\}$ of $E$ consisting of eigenvectors of $Q,\, Qe_k=\lambda e_k$, and the corresponding
 orthonormal basis of $H=Q^{1/2}(E)$ is $\mathcal V=\{v_k:=\sqrt{\lambda_k}e_k:k\in\mathbb N\}$. For each $k$
 the function $\hat v_k$ is just $\hat v_k(x)=\frac{x_k}{\sqrt\lambda_k}$, where $x_k=\langle x,e_k\rangle_X$.

 For every $r>0$ the function $G(x):=\|x\|^2-r^2$ satisfies Hypothesis \ref{Hypo}. Indeed, it is smooth,
 $\mo=B(0,r),\, \Dho G(x)=2Qx$ and $1/|\Dho G|_H=1/2\|Q^{\frac{1}{2}}x\|$ is easily seen to belong to
 $L^p(E,\gamma)$ for every $p$.

 Then for $\varphi\in W^{1,2}(B(0,r),\gamma)$ the integration by parts formula reads
 \[
  \int_{B(0,r)} D_k^{\mo}\varphi\, d\gamma=\frac{1}{\sqrt{\lambda_k}}\int_{B(0,r)}x_k\varphi\, d\gamma+
  \int_{\|x\|=r}\frac{\sqrt{\lambda_k}x_k}{\|Q^{1/2}\|x}\varphi \, d\rho,
 \]

 Consequently, the componentewise Skorohod decomposition reads
 \[
  X_t^k=z+W_t^k+\frac{1}{\sqrt{\lambda_k}}\int_0^tX_s^k\,ds+\int_0^t \frac{\sqrt{\lambda_k}X_s^k}{\|Q^{1/2}X_s\|} \,
  dL_s^{\rho}
 \]
 and the Skorohod decomposition of the infinite dimensional reflecting Ornstein-Uhlenbeck process $(X_t)_{t\geq 0}$ is given by
 \[
  X_t=z+W_t+\int_0^tX_s\,ds+\int_0^t \frac{QX_s}{\|Q^{1/2}X_s\|} \,
  dL_s^{\rho}
 \]

\par\bigskip

\subsection*{Acknowledgments}
It is my great pleasure to acknowledge fruitful and stimulating
discussions with Michael R\"{o}ckner on the topics discussed in this
paper. I warmly thank M. Kunze and M. sauter and the all group of W.
Arendt in Ulm (germany), where the most ideas of this paper was
discussed.


\begin{thebibliography}{00}


\bibitem{AR} Albeverio S. and Roeckner M.:{\it Stochastic Differential Equations in Infinite Dimensions:
Solutions via Dirichlet Forms}, Probab. Th. Rel. Fields 89, 347-386, 1991.

\bibitem{AKR} Albeverion S., Kusuoka S., R\"ockner M., {\it On partial integration in infinite dimensional space
and applications to Dirichlet forms.} J. Lond. Math. Soc. {\bf 42}, 122-136, 1990.


\bibitem{AMP} Ambrosio L., Miramda M., Pallara D., {\it Sets of finite perimeter in Wiener spaces, perimeter measure
and boundary rectifiability}, Discr. Cont. Dynam. Systems {\bf 28} (2010), 591-608.

\bibitem{AW} Arendt W. and Warma M.:{\it The Laplacian with Robin Boundary Conditions on Arbitrary Domains},
Potential Anal. 19, 341-363, 2003.

\bibitem{BPT} Barbu V., Da Prato G., Tubaro L.: {\it Kolmogorov equation associated to the stochastic reflection
problem on a smooth convex set of a Hilbert space,} Ann. probab. {\bf 37} (2009), 1427-1458.

\bibitem{BPT2} Barbu V., Da Prato G., Tubaro L.: {\it Kolmogorov equation associated to the stochastic reflection
problem on a smooth convex set of a Hilbert space II,} Ann. Inst. H. Poincar\'e Probab. Stat. {\bf 47} (2011),
699-724.

\bibitem{BHs} Bass R. F., Hsu P.: {\it The semi­martingale structure of reflecting Brownian motion,}
Proc. Amer. Math. Soc. 108 {\bf 4} (1990), pp. 1007-1010.

\bibitem{BH} Bouleau, N., Hirsch F. :Dirichlet Forms and Analysis on Wiener Space.{\em Walter de Gruyter}, Berlin, 1991.

\bibitem{Bo} Borgachev V.I. : Gausssian Measures, vol. 62 of Mathematical Surveys and Monographs, American
Mathematica Society, Providence, 1998.

\bibitem{Ch} Choquet G.: {\it Theory of capacities,} Ann. Inst. Fourier, Grenoble, {\bf 5} (1953-1954), pp.131-295.

\bibitem{CL} Celada P. and Lunardi A.: Traces of Sobolev Functions on Regular Surfaces in Infinite Dimensions,
Arxiv:1302.2204v1, 2013.

\bibitem{Da} Da Prato G., {\it An introduction to infinite dimensional analysis,} Springer-Verlag, Berlin 2006.

\bibitem{Fe} Feyel D., {\it Hausdorff-Gauss measures,} in: Stochastic Analysis and Related Topics, VII.
Kusadasi 1998, Progr. in Probab. 98, Birkh\"{a}user, Boston 2011,59-76.

\bibitem{FL} Feyel D., de La Pradelle A., {\it Hausdorff measures on the Wiener space}, Pot. Analysis {\bf 1} (1992),
177-189.

\bibitem{F80} Fukushima M.: Dirichlet Forms and Markov Processes.
{\em Amsterdam: North Holand}, (1980).

\bibitem{FH} Fukushima M., Hino M.: {\it On the Space of BV Functions and a Related Stochastic
Calculus in Infinite Dimensions.} Journal of Functional Analysis, {\bf183}, 245-268 (2001)

\bibitem{Hi} Hino M., {\it Sets of finite perimeter and the Hausdorff-Gauss measure on the Wiener space,} J. Funct.
Anal. {\bf 258} (2010), 1656-1681.

\bibitem{Ho} Hohmann, R.:{\it Stetigkeitsbedingungen f\"ur stochastische Prozesse.} Diplomarbeit Universitfit
Bielefeld 1985.

\bibitem{Ku} Kuo, H.:  Gaussian measures in Banach spaces. (Lect. Notes Math., vol. 463,
pp. 1-224) Berlin Heidelberg New York: Springer 1975.

\bibitem{KS} K\"{u}nze M., Sauter M., {\it Relative Gaussian Capacity}, in preparation.

\bibitem{MR} Ma Z. M. R\"ockner M.: Introduction to the Theory of (non-symmetric) Dirichlet Forms, Universitex,
Spinger-Verlag, Berlin, 1992.

\bibitem{Ro} R\"ockner, M.: {\it Traces of harmonic functions and a new path space for the free
quantum field}. J. Funct. Anal. {\bf 79}, 211-249 (1988)

\bibitem{RZZ} R\"ockner M., Zhu R-C., Zhu X-C.: {\it The stochastic reflection problem on an infinite dimensional
convex set and BV functions in a Gelfand triple.} The Annals of Probability, Vol. 40, No. 4 (2012), 1759-1794.

\bibitem{Str} Strook, D.W.: An introduction to the theory of large deviations.{\it  New York
Berlin Heidelberg: Springer}, 1984.

\bibitem{St} Stolmann P.:{\it Closed ideals in Dirichlet spaces}, Potential Anal., {\bf 2} (1993), pp. 263-268.


\end{thebibliography}
\end{document}